\documentclass[11pt]{amsart}

\usepackage[backend=biber,style=alphabetic]{biblatex}
\usepackage{amsmath, amssymb, amsthm, mathtools, mathrsfs}

\usepackage[margin=3cm]{geometry}

\newcommand{\abs}[1]{\lvert {#1} \rvert}
\newcommand{\fanc}{\mathscr{C}}
\newcommand{\deff}[1]{\textbf{#1}}
\newcommand{\N}{\mathbb{N}}

\newcommand{\F}{\mathbb{F}}

\DeclarePairedDelimiter\ceil{\lceil}{\rceil}
\DeclarePairedDelimiter\floor{\lfloor}{\rfloor}

\DeclareMathOperator{\exc}{X}
\DeclareMathOperator{\qc}{QC}
\DeclareMathOperator{\aff}{aff}
\DeclareMathOperator{\arank}{arank}
\DeclareMathOperator{\mult}{mult}
\usepackage{setspace}
\usepackage{graphicx}
\usepackage{color}
\usepackage{hyperref}
\usepackage{verbatim}
\graphicspath{ {./images/} }
\usepackage{enumitem}
\newtheorem{theorem}{Theorem}[section]
\newtheorem{proposition}[theorem]{Proposition}
\newtheorem{lemma}[theorem]{Lemma}
\newtheorem{corollary}[theorem]{Corollary}

\theoremstyle{definition}
\newtheorem{definition}[theorem]{Definition}

\theoremstyle{remark}
\newtheorem{remark}[theorem]{Remark}
\newtheorem{example}[theorem]{Example}

\usepackage[parfill]{parskip}

\usepackage[justification=centering]{caption}

\title[Caps in $AG(2,n)$]{How many cards should you lay out in a game of \textit{EvenQuads}: A detailed study of caps in $AG(n,2)$}

\author[Crager]{Julia Crager}
\address{Department of Mathematics, Bard College, Annandale-on-Hudson, New York 12504}
\email{jc9281@bard.edu}

\author[Flores]{Felicia Flores}
\address{Department of Mathematics, Bard College, Annandale-on-Hudson, New York 12504}
\email{ff0811@bard.edu}

\author[Goldberg]{Timothy E.~Goldberg}
\address{Department of Mathematics, Lenoir-Rhyne University, Hickory, North Carolina 28601}
\email{timothy.goldberg@lr.edu}

\author[Rose]{Lauren L.~Rose}
\address{Department of Mathematics, Bard College, Annandale-on-Hudson, New York 12504}
\email{rose@bard.edu}

\author[Rose-Levine]{Daniel Rose-Levine}
\address{Department of Mathematics, Bard College, Annandale-on-Hudson, New York 12504}
\email{dr6048@bard.edu}

\author[Thornburgh]{Darrion Thornburgh}
\address{Department of Mathematics, Bard College, Annandale-on-Hudson, New York 12504}
\email{dt9275@bard.edu}

\author[Walker]{Raphael Walker}
\address{Institut de Math\'ematique d'Orsay, Universit\'e Paris-Saclay, Orsay, France}
\email{raphael.walker@universite-paris-saclay.fr}

\date{\today}

\begin{document}

\begin{abstract}
    We define a \emph{cap} in the affine geometry $AG(n,2)$ to be a subset in which any collection of 4 points is in general position. In this paper we classify, up to affine equivalence, all caps in $AG(n,2)$ of size $k \leq 9$. As a result, we obtain a complete characterization of caps in dimension $n \leq 6$, in particular complete and maximal caps. Since the \textit{EvenQuads} card deck is a model for $AG(6,2)$, as a consequence we determine the probability that an arbitrary $k$-card layout contains a quad.
\end{abstract}

\maketitle

\tableofcontents

\section{Introduction}


The card game Quads, initially called SuperSET, was introduced by Rose and Perreira in \cite{Rosepereira2013Super}. It is a pattern recognition game, similar to the popular card game SET$\textsuperscript{\textregistered}$. The mathematical underpinnings of Quads are broad and can be explored at a variety of levels of background and experience, using tools from combinatorics, probability, linear and affine algebra, abstract algebra, design theory, and finite geometry \cite{rose-chapter}.


\begin{figure}[htbp]
    \centering
    \includegraphics[width=10cm]{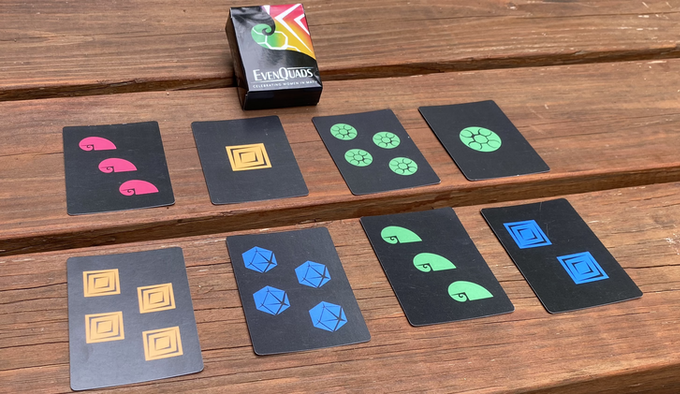}
    \caption{A Layout of EvenQuads cards. Photo by Denise Rangel Tracy.}
    \label{fig:EvenQuads}
\end{figure}

A Quads deck was published in 2021 by the Association for Women in Mathematics under the name \emph{EvenQuads} \cite{EvenQuads}, in honor of its 50th anniversary. See Figure~\ref{fig:EvenQuads}. 
In this deck, each symbol is the logo of a different mathematics society, and on the reverse side are profiles of prominent women mathematicians.
Eventually, there will be four different \emph{\textit{EvenQuads}} decks, which can be purchased on the website \href{https://store.awm-math.org/products/regular-edition-EvenQuads-notable-women-in-math-playing-cards}{https://store.awm-math.org/products/regular-edition-EvenQuads-notable-women-in-math-playing-cards}.
One can also play the game online using the \emph{Quads App} \cite{Roseapp}.

This paper serves as both a broad introduction to the mathematics behind \emph{Quads} and a detailed study of collections of cards that don't contain a quad. These sets are called \deff{$2$-caps} in \cite{Bennett-bounds} and \deff{Sidon Sets} in \cite{taitwon2021}.

Quad-free sets in Quads are analogous to \emph{caps} in the card game SET. A \emph{cap} or a \emph{cap set} in the affine geometry $AG(n,q)$ is defined to be a subset in which every collection of $q$ points is in general position. A SET deck is a model for the affine geometry $AG(4,3)$, and it can easily be generalized to $AG(n,3)$ by adding additional attributes. A SET in the game corresponds to a line in $AG(n,3)$, which consists of three collinear points, and thus a cap is a SET-free collection of cards.

Caps in $AG(n,3)$ have been well studied, and there is great interest in determining the maximal size of a cap in $AG(n,3)$. This is known as the \emph{Cap Set} problem. The maximal cap size is currently only known for $n \leq 6$ and is otherwise open. However, recently Ellenberg and Gijswijt \cite{ellenberg2016large} found a much improved upper bound on maximal cap sizes,  $(2.756)^n$, making use of a technique called the \emph{polynomial method} in a new way \cite{croot-progression-free}. 
See \cite{davis-maclagan} or \cite[Chapter 9]{Rosemcmahon2016joy} for a detailed introduction to the \emph{Cap Set} problem. This celebrated open problem both informed and inspired our work.

It turns out that a quads deck is a model for the affine geometry $AG(6,2)$, and it can easily be generalized to $AG(n,2)$. Because the only caps in $AG(n,2)$ are the singleton sets, for the purpose of studying quads we redefine a \emph{cap in $AG(n,2)$}to be a subset in which every collection of \emph{four} points is in general position. Thus, caps in $AG(n,2)$ will correspond to quad-frees set in a quads deck.

Although caps in $AG(n,2)$ have not been well studied, there are several important results in the literature. In \cite[Theorem 5.1]{taitwon2021}, Tait and Won found asymptotic bounds for the maximal size $M(n)$ of a cap in the affine geometry $AG(n,2)$:
\[
    \frac{1}{\sqrt{2}} \, 2^{n/2}
    \leq M(n)
    \leq 1 + \sqrt{2} \cdot 2^{n/2}.
\]
In \cite{RedmanRoseWalker}, Redman, Rose, and Walker constructed a small complete cap in $AG(n,2)$ for each $n$, providing an upper bound on the size of a minimal complete cap, A \emph{complete cap}, which is a cap such that the addition of any new element will create a quad, need not be of maximal size.

In this paper we classify all caps of size $k$ in $AG(n, 2)$ for $k \leq 9$, including complete and maximal caps, up to affine equivalence. We also provide structural theorems for caps of arbitrary size with at most one affine dependence. As an application of the classification of caps in small dimensions, we compute the probability that a $k$-card layout of cards contains a quad, providing an answer to the question in the title: 

\begin{quote}
\large{\textit{How many cards should you lay out
in a game of {EvenQuads}?}}
\end{quote}

\begin{figure}[htbp]
    \centering
    \includegraphics[width=10cm]{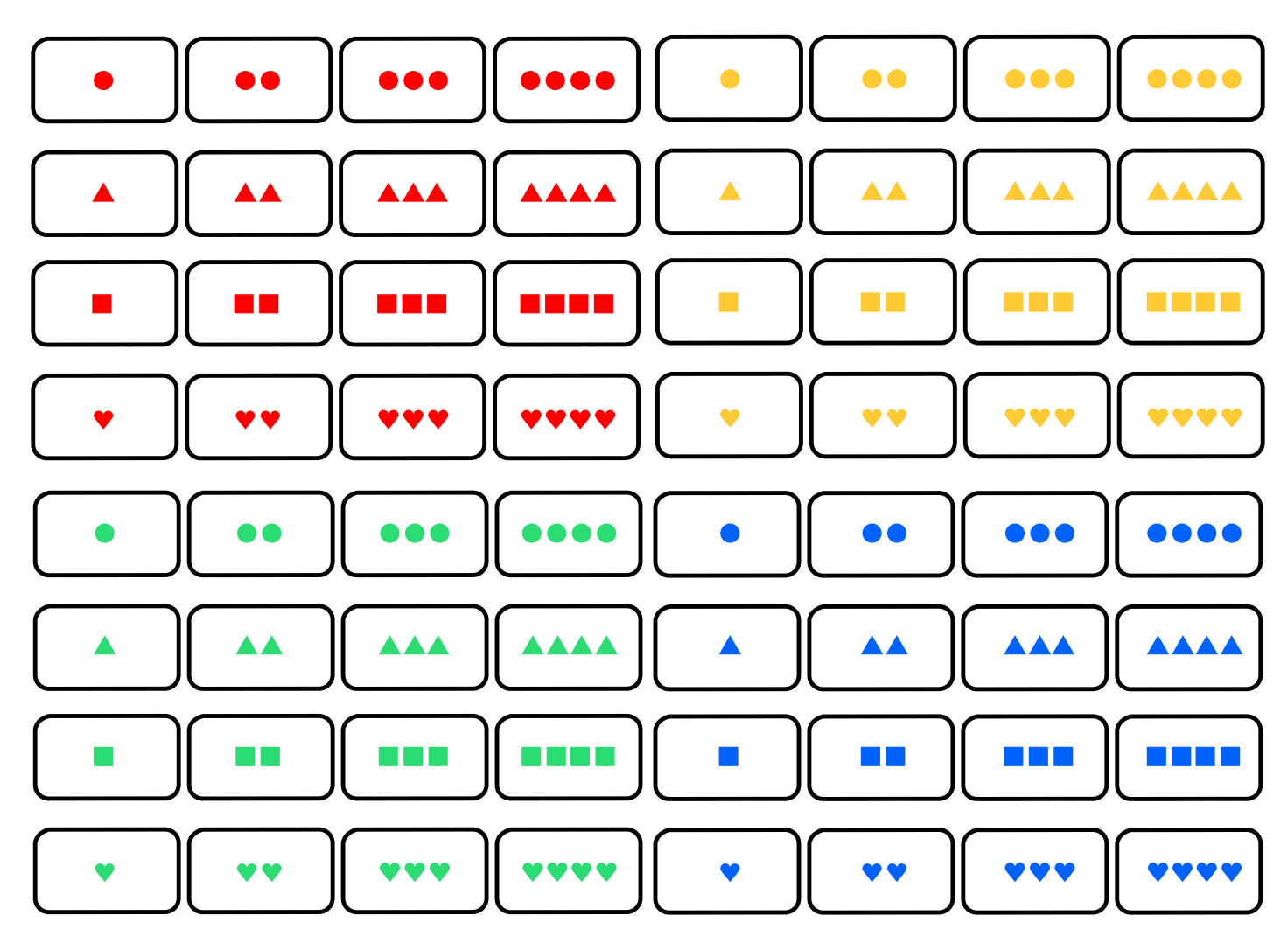}
    \caption{A Quads Deck. Design by Julia Crager.}
    \label{fig:quads-deck}
\end{figure}

\subsection{Quads Gameplay}

Quads is played with a special deck of cards, similar to SET. The deck we use in this paper is pictured in Figure~\ref{fig:quads-deck}. Each card has a design with three attributes --- number, color, and shape --- and each attribute can be in exactly one of \textbf{four} possible states, given in Table~\ref{tab:values}.

\begin{table}[ht]
    \centering
    \caption{The three attributes and their states in Quads.}
    \label{tab:values}
    \begin{tabular}{|l|l|}
        \hline
        \textbf{Attribute} & \textbf{States} \\
        \hline
        Number & 1, 2, 3, 4 \\
        \hline
        Color & Red, Green, Yellow, Blue \\
        \hline
        Shape & Squares, Triangles, Circles, Hearts \\
        \hline
    \end{tabular}
\end{table}

A standard Quads deck has 64 cards, one for each combination of states and attributes. The goal of the game is to find sets of 4 cards satisfying the {\emph Quad Conditions} given below. 
\bigskip

\begin{center}
\boxed{
\begin{minipage}{10cm}
    \textbf{Quad Conditions:} A set of 4 cards forms a quad if in each attribute one of the following holds.
    \begin{enumerate}
        \item The states are the same on each card.
        \item The states on each card are all different.
        \item Two different states occur, each on two cards.
    \end{enumerate}
\end{minipage}
}
\end{center}

\bigskip

\begin{example}
In Figure~\ref{fig:examples}, each of the first three horizontal sets is a quad. In the first set, the cards have the same number of objects (1), the same color (red), and different shapes. The second has different numbers, the same shape (circles), and two pairs each of two colors. The third has two pairs of shapes, and different numbers and colors. However, the fourth fails the Quad Conditions in the attribute shape, because two cards contain circles and the other two have different shapes. 
\end{example}

\begin{figure}[htbp]
    \centering
    \begin{tabular}{c}
        \includegraphics[width=7cm]{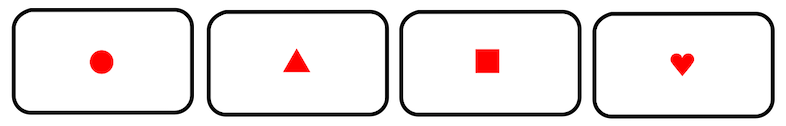} \\[2ex]
        \includegraphics[width=7cm]{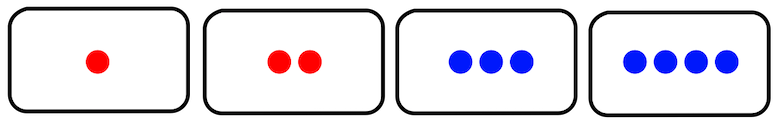} \\[2ex]
        \includegraphics[width=7cm]{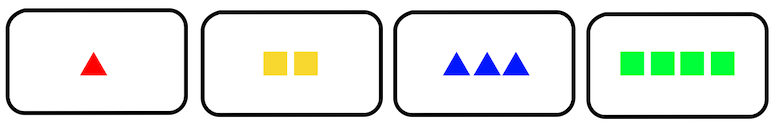} \\[2ex]
        \includegraphics[width=7cm]{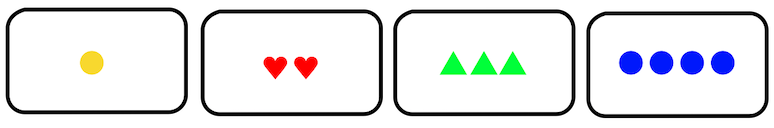}
    \end{tabular}
    \caption{Three quads and a non-quad.}
    \label{fig:examples}
\end{figure}

The first two Quad Conditions are identical to those for SET, but the third is required in order to make sure that any set of 3 cards can be completed uniquely to a quad.  To see this, consider the second quad in Figure~\ref{fig:examples}. The first three cards make it impossible to have all the same or all different colors. This third condition makes a Quads deck a \emph{Steiner Quadruple System}, analogous to a SET deck being a \emph{Steiner Triple System}.

Quads is played by placing a random layout of cards face up on a table. Players race to identify and pick up quads from these cards. Players collects quads into their own piles, and each time 4 new cards are added to the layout from the deck. This process continues until the deck is empty and no quads remain. If at any point the players all agree there is no quad in the layout, an additional card is added. The player who collects the most cards wins.

\subsection{Card layouts in SET and Quads}

A key detail of the rules is how many cards to lay out. The official rules for SET suggest an initial 12 card layout. In order to determine the this number, expected value computations can been used to determine the average number of SETs in a $k$-card layout, and computer trials can be used to approximate the probability of a SET in a random $k$-card layout, as the numbers are too large for brute force computions. \cite[Chapter 10]{Rosemcmahon2016joy}.

In the case of Quads, we are not only able to compute the exact probability of a quad in a $k$-card layout, but we can also derive these results using algebraic and combinatorial arguments that shed light on the underlying mathematical structures.

The main goal of this paper is to determine the affine equivalence classes of caps of size $k$ in $AG(n,2)$. This research was motivated by both the 
\emph{Cap Set} and the desire to determine how many cards to lay out in a game of Quads. For the latter problem, we begin by classifying caps up to affine equivalence, then count the number in each equivalence class, and use these results to compute the probability that a random $k$-card layout contains a quad. Our findings are summarized in Table~\ref{tab:quad-probs2}. 

\begin{table}[htbp]
\centering
\caption{The probabilities that $k$ {\textit{EvenQuads}} cards contain a quad.}
\label{tab:quad-probs2}
\[
\begin{array}{|c|c|}
    \hline
    k & \text{Probability of quad} \\
    \hline \hline
    1 & 0\% \\
    \hline
    2 & 0\% \\
    \hline
    3 & 0\% \\
    \hline
    4 & 1.64\% \\
    \hline
    5 & 8.20\% \\
    \hline
    6 & 23.76\% \\
    \hline
    7 & 49.78\% \\
    \hline
    8 & 79.04\% \\
    \hline
    9 & 96.39\% \\
    \hline
    10 & 100\% \\
    \hline
\end{array}
\]
\end{table}

Informed by this table, players may choose to lay out either 8, 9, or 10 cards, depending on their preferences. While 10 cards guarantee a quad, and with 9 cards there will be a quad most of the time, with 8 cards the challenge is either to find a quad or to show that none exists.  

Equivalence classes of maximal caps in $AG(n,3)$ have been studied, for example in \cite{max-cap-partitions}, but we haven't seen anything in the literature on non-maximal or incomplete caps.

\subsection{Overview}

In Section~\ref{sec:coordinates}, we describe in detail the correspondence between Quads cards and elements of $\mathbb{Z}_2^n$.

In Section~\ref{sec:caps}, we give precise definitions of caps and related notions, and introduce a useful online app called the \emph{Qap Visualizer}.

In Section~\ref{sec:affinegeo} we provide an introduction to affine geometry and the basic properties we will need in subsequent sections.

In Section~\ref{sec:geomprops}, we view quads and caps in terms of affine geometry, completely characterizing caps of sizes 1--7.

In Section~\ref{sec:equivalenceclasses}, we define equivalence classes of caps, prove two general theorems, and then give a complete characterization of caps of sizes 1--9. This covers all caps in $AG(6,2)$ and hence in an \textit{EvenQuads} deck.

In Section~\ref{sec:countingcaps}, we develop general formulas to count the number of caps of sizes 1--9 in each equivalence class. We then apply these formulas to caps in $AG(6,2)$, enabling us to compute the probability that a $k$-card layout in {\emph{\textit{EvenQuads}}} contains a quad.

\section{Coordinates for Quads} \label{sec:coordinates}

Just as the cards in a SET deck can be assigned coordinates in the finite vector space $\mathbb{Z}_3^4$ in a way compatible with the structure of the game, the same can be done with a Quads deck and $\mathbb{Z}_2^6$. 

\subsection{Quad cards as binary vectors}

As a first step, for each attribute we map the four states to the elements of the group $\mathbb{Z}_2 \times \mathbb{Z}_2$, which we denote as binary strings: $\{ 00, 01, 10, 11 \}$. The value assigned to each state is called its \deff{$(\mathbb{Z}_2 \times \mathbb{Z}_2)$-coordinate}. For example, we could use the assignments given in the following table.

\begin{table}[ht]
    \centering
    \caption{An assignment of $(\mathbb{Z}_2 \times \mathbb{Z}_2)$-coordinates to states of attributes of Quads cards.}
    \label{tab:coords}
    \begin{tabular}{|l|l|l|} 
    \hline
     \textbf{Attribute} & \textbf{States} & \textbf{$(\mathbb{Z}_2 \times \mathbb{Z}_2)$-Coordinates} \\ 
    \hline
    Number & $1,2,3,4$ & $00,01,10,11$ \\ 
    \hline
    Color & Green, Red, Blue, Yellow & $00,01,10,11$ \\
    \hline
    Shape & Heart, Square, Triangle, Circle & $00,01,10,11$ \\ 
    \hline
    \end{tabular}
\end{table}

In this way, each card in a Quads deck corresponds to an element of $(\mathbb{Z}_2 \times \mathbb{Z}_2)^3$. For example, the card four-green-triangles corresponds to the element $(11,00,10)$. Using the group isomorphism $(\mathbb{Z}_2 \times \mathbb{Z}_2)^3 \cong \mathbb{Z}_2^6$, each card corresponds to an element of the vector space $\mathbb{Z}_2^6$ over the field $\mathbb{Z}_2$, called the card's \deff{$\mathbb{Z}_2$-coordinates}, for example $(11,00,10) \rightarrow (1,1,0,0,1,0)$.  
This vector space, and hence the Quads deck, forms a model for the finite affine geometry $AG(6,2)$. For a generalized Quads deck with $n$ attributes, the cards can be put in one-to-one correspondence with elements of $(\mathbb{Z}_2 \times \mathbb{Z}_2)^n \cong \mathbb{Z}_2^{2n}$, which is a model for $AG(2n,2)$.

We will refer to the standard deck of $64$ cards, with three attributes, as \textbf{Quad-64}, and the generalized game with $n$ attributes as \textbf{Quad-$(4^n)$}. The crucial property of these coordinates is the following result, analogous to the corresponding property of $\mathbb{Z}_3$-coordinates for SET cards as in \cite[page 3]{davis-maclagan}.

\begin{theorem} \label{thm:fundamental}
Let $a,b,c,d \in \mathbb{Z}_2^6$ be distinct. Then $\{a,b,c,d\}$ is a quad if and only if $a+b+c+d=\vec{0}$.

\begin{proof}
    This follows from the fact that if $w,x,y,z \in \mathbb{Z}_2 \times \mathbb{Z}_2$, then $w+x+y+z=00$ if and only if one of the following is true.
    \begin{itemize}
        \item The elements $w,x,y,z$ are all equal.
        \item The elements $w,x,y,z$ are all distinct.
        \item The list of elements $w,x,y,z$ consists of two distinct elements repeated twice.
    \end{itemize}
\end{proof}
\end{theorem}

Recall that since $\mathbb{Z}_2^6$ is a vector space over $\mathbb{Z}_2$, for any $a \in \mathbb{Z}_2^n$ we have $a+a=\vec{0}$ and $a = -a$. The following corollary assures us that not only can three cards always be completed to a quad, but the fourth card is unique. 

\begin{corollary} \label{fundcor}
Let $a,b,c \in \mathbb{Z}_2^6$ be distinct. Then there is exactly one quad containing these elements, and the fourth element is $d = a+b+c$.

\begin{proof}
    Let $d = a+b+c$. First, note that $d$ cannot equal any of $a,b,c$, because if $d=a$, for example, then $d = d+b+c$, so $b+c=\vec{0}$, so $b=c$, which contradicts the assumption that $a,b,c$ are distinct. Hence $a,b,c,d$ are distinct elements of $\mathbb{Z}_2^n$.
    Now, note that $\{a,b,c,d\}$ is a quad by Theorem~\ref{thm:fundamental}, because
    \[
        a+b+c+d
        = a+b+c+(a+b+c)
        = (a+a)+(b+b)+(c+c)
        = \vec{0}.
    \]
    Finally, suppose $d' \in D$ and $\{a,b,c,d'\}$ is a quad. Then by Theorem~\ref{thm:fundamental} we know that $a+b+c+d'=\vec{0}$. Adding $d'$ to both sides yields $a+b+c = d'$. Therefore  $d $ is unique, and exactly one quad contains $a,b,c$.
\end{proof}
\end{corollary}

The next corollary follows immediately from Corollary~\ref{fundcor}, and provides a useful way for computing the sum of two cards without having to recall the entire coordinate assignment --- just which card corresponds to the zero vector. Since the sum of an element with itself is $\vec{0}$ and the sum of any element with $\vec{0}$ is that element, we need only consider the sum of distinct nonzero elements.

\begin{corollary} \label{easy_addition}
Let $a,b \in \mathbb{Z}_2^6$ be distinct nonzero vectors. Then $a+b$ is the unique card that makes a quad with $\vec{0}$, $a$, and $b$.
\end{corollary}

It is clear from the proofs that Theorem~\ref{thm:fundamental} and Corollaries~\ref{fundcor} and~\ref{easy_addition} are true in a deck with $n$ attributes. In fact, one can also play \textbf{Quad-$2^n$} by using what we call \emph{half-attributes}. For example, if we consider the Quad-64 cards containing red or blue symbols, shown in Figure~\ref{fig:quad32}, these 32 cards can be viewed as vectors in $\mathbb{Z}_2^5$. It's easy to see that this set is closed under taking quads, i.e.~for any three cards, the fourth card that completes a quad is still in the set.

\begin{figure}[htbp]
    \centering
    \includegraphics[width=10cm]{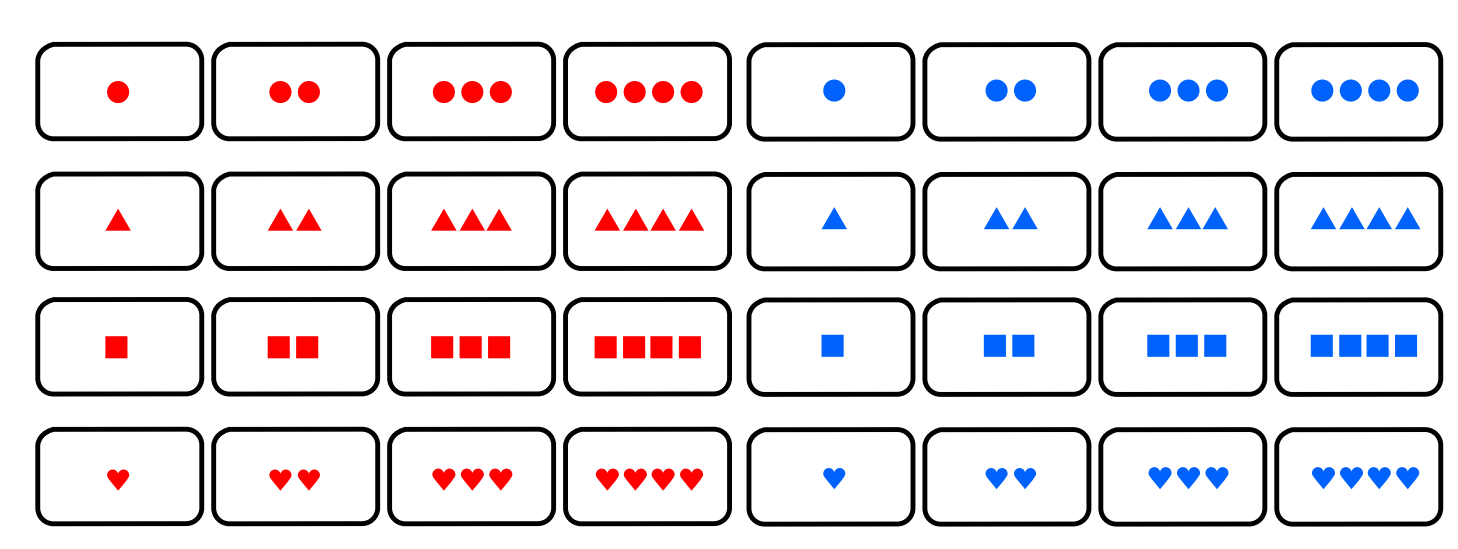}
    \caption{A possible Quad-32 deck, which can be represented by $\mathbb{Z}_2^5$.}
    \label{fig:quad32}
\end{figure}

\subsection{Visualizing coordinates in Quads}

Viewed as elements of the Euclidean space $\mathbb{R}^6$, the vectors in $\mathbb{Z}_2^6$ are the vertices of a $6$-dimensional hypercube.  However, for our purposes it is more convenient to visualize elements of $\mathbb{Z}_2^6$ and their coordinates in a planar fashion, in an $8 \times 8$ grid. We assign coordinates to the 64 squares in the grid using a method we call \deff{recursive coordinates}. To begin with, we label a $2 \times 2$ grid with elements of $\mathbb{Z}_2 \times \mathbb{Z}_2$ as as in Figure~\ref{grid-coords}. 

\begin{figure}[ht]
    \centering
    \[
    \begin{array}{|c|c|}
        \hline
        00 & 01 \\
        \hline
        10 & 11 \\
        \hline
    \end{array}
    \]
    \caption{Grid layout of the elements of $\mathbb{Z}_2 \times \mathbb{Z}_2$.}
        \label{grid-coords}
\end{figure}

To build coordinates in the full $8 \times 8$ grid, first divide this into four $4 \times 4$ subgrids and assign an element of $\mathbb{Z}_2 \times \mathbb{Z}_2$ to each using the scheme from Figure~\ref{grid-coords}. This will be the first two entries of the $\mathbb{Z}_2$-coordinates of any square contained in that $4 \times 4$ grid. For example, the $\mathbb{Z}_2$-coordinates of a square in the upper-right $4 \times 4$ grid has coordinates of the form $(0,1,\ldots)$.

Then divide each $4 \times 4$ grid into four $2 \times 2$ grids, each of which is assigned an element of $\mathbb{Z}_2 \times \mathbb{Z}_2$ in the same way. These assignments determine the third and fourth entries of the $Z_2$-coordinates. Finally, each square in a $2 \times 2$ grid is assigned an element of $\mathbb{Z}_2 \times \mathbb{Z}_2$, which determines the last two entries of its $Z_2$-coordinates. We now give an example of this.

\begin{example} the coordinates of the marked box in Figure~\ref{fig:coords-example},  are $(1,1,0,0,1,0)$, because it is in the bottom right $4 \times 4$ grid ($11$), within this the upper left $2 \times 2$ grid($00$), and within this the upper right square ($01$).
\end{example}

\begin{figure}[htbp]
    \centering
    \includegraphics[width=6cm]{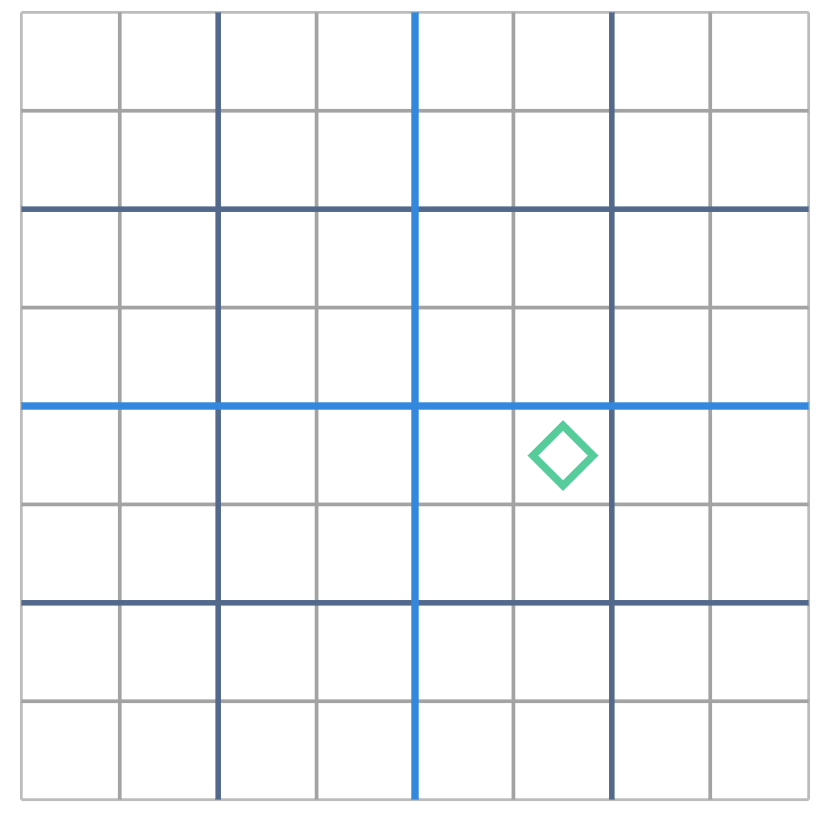}
    \caption{The point in $\mathbb{Z}_2^6$ with recursive coordinates $(1,1,0,0,0,1)$.}
    \label{fig:coords-example}
\end{figure}

In this way, we can place the $64$ elements of $\mathbb{Z}_2^6$ in the grid according to their $\mathbb{Z}_2$-coordinates. In practice, we can restrict our attention to squares in the grid rather than Quad cards, since the fundamental properties of the cards are completely captured by their coordinates.

The same method can be used to form planar diagrams of a deck of Quads cards with $n$ attributes, using a $2^n \times 2^n$ grid.

\section[Caps in AG(n,2)]{Caps in $AG(n,2)$} \label{sec:caps}

In this section, we define basic terms and explore properties of caps using  $\mathbb{Z}_2^n$ as a model for $AG(n,2)$. We also introduce the \emph{Qap Visualizer}, an online tool for viewing, constructing, and analyzing caps. 

\subsection{Cap definitions and properties}

We now define the notions of quad completion and quad closure.

\begin{definition} Let $S$ be a subset of $\mathbb{Z}_2^n$.
\begin{enumerate}
    \item Let $a,b,c \in \mathbb{Z}_2^n$ be distinct. The \deff{quad completion} of $a,b,c$ is $d = a+b+c$. 
    \item The \deff{exclude set} of $S$, denoted $\exc(S)$, is defined by 
    \[
        \exc(S) = \{ a+b+c \mid \text{$a,b,c \in S$ are distinct} \}.
    \]
    An element of $\exc(S)$ is called an \deff{exclude point} of $S$.

    \item The \deff{quad closure} of $S$ is the union of $S$ with its exclude set, $\qc(S) = S \cup \exc(S)$.

    \item An element $p \in \exc(S)$ has \deff{multiplicity} $m$, and is called an \deff{$m$-point} of $S$, if there are exactly $m$ distinct (unordered) triples of elements of $S$,
    \[
        \{x_i, y_i, z_i\} \subseteq S \qquad \text{for $i=1,\ldots,m$},
    \]
    such that $x_i+y_i+z_i=p$ for each $i$.
    
    \item The set $S$ is a \deff{cap} if $S$ does not contain a quad, i.e.~for any distinct $a,b,c,d \in S$, $a+b+c+d \neq \vec{0}$. 
    A cap with $k$ elements is called a \deff{$k$-cap}.
    
    \item A cap $S$ is \deff{complete} if $\exc(S) = \mathbb{Z}_2^n - S$.
    
    \item A cap is \deff{maximal} if it is of maximal size in $\mathbb{Z}_2^n$.
    
\end{enumerate}
\end{definition}

The concept of a complete or maximal cap generalizes easily to the situation where $\mathbb{Z}_2^n$ is replaced by any $r$-dimensional affine subspace of $\mathbb{Z}_2^n$. This will be made explicit in the next section.

We state some useful properties of caps.

\begin{proposition} \label{basic-props}
Let $S$ be a subset of $\mathbb{Z}_2^n$.
\begin{enumerate}
    \item If $S$ has fewer than four elements, it is a cap.
    \item $S$ is a cap if and only if $S \cap \exc(S) = \emptyset$.
    \item $S$ is a complete cap if and only if $\qc(S) = \mathbb{Z}_2^n$.
    \item Maximal caps are complete.
    \item If all $k$-caps in dimension $n$ are complete, then they are maximal.
    \item \label{sum-of-types} If $S$ is a $k$-cap, then
    \[
        \sum_{p \in \exc(S)} \mult(p) = \binom{k}{3}.
    \]
\end{enumerate}

\begin{proof}
Statements 1--4 follow directly from the definitions. 

For statement 5, suppose not all $k$-caps are maximal. This means that there exists a maximal $m$-cap $C'$, where $m > k$. Now remove points until you have a $k$-cap. This cap will not be complete, since you can add points and still have a cap, contradicting our assumption that all $k$-caps are complete. 

For statement 6, suppose $S$ is a $k$-cap. Let $\mathcal{T}$ be the set of all $3$-element subsets of $S$. We can map each triple $\{a,b,c\} \in \mathcal{T}$ to the element $p = a+b+c \in \exc(S)$, and for each $p \in \exc(S)$ the multiplicity of $p$ is exactly the number of triples in $\mathcal{T}$ that map to $p$. It follows that
\[
    \sum_{p \in \exc(S)} \mult(p)
    = \lvert \mathcal{T} \rvert
    = \binom{k}{3}.
\]
\end{proof}
\end{proposition}

Statement 2 of Proposition~\ref{basic-props} says that if we want to enlarge a cap $S$, we cannot use any elements from $\exc(S)$. i.e.~points in $\exc(S)$ must be \textbf{excluded} from any expanded cap.

We have a useful fact about $m$-points. Not only must the $m$ triples that determine the point be distinct, but they cannot have any overlap.

\begin{proposition} \label{disjoint-triples}
Let $C \subseteq \mathbb{Z}_2^n$ be a cap, and let $p$ be an $m$-point of $C$ with $m \geq 2$. Then any two triples $\{a, b , c\}, \{d, e, f\} \subseteq C$ that sum to $p$ must be disjoint.

\begin{proof}
Let $\{a,b,c\}, \{d,e,f\} \subseteq C$ be distinct triples such that $a + b + c = d + e + f$. If the triples overlap in one point, say $a = d$, then $a + b + c = a + e + f$, so $b + c + d + e = \vec{0}$, so these four points form a quad, contradicting the fact that they come from a cap. If they overlap in two points, say $a = d$ and $b = e$, then $a + b + c = a + b + f$, so $c = f$, contradicting the assumption that triples are distinct. Hence the triples must be disjoint.
\end{proof}
\end{proposition}

Proposition~\ref{disjoint-triples} provides a connection between $m$-points and $k$-caps.

\begin{corollary}
\label{cor:s-geq-3k}
Let $C \subseteq \mathbb{Z}_2^n$ be a $k$-cap. If $\exc(C)$ contains an $m$-point, then $k \geq 3m$.

\begin{proof}
Suppose $p \in \exc(C)$ is an $m$-point. Then by definition there are exactly $m$ distinct triples $\{x_i,y_i,z_i\} \subseteq C$, $i = 1, \ldots, m$ for which $x_i + y_i + z_i = p$. By Corollary~\ref{disjoint-triples}, these triples are pairwise disjoint, which means the $3m$ elements
\[
    x_1, y_1, z_1, \ldots, x_k, y_k, z_k \in C
\]
are distinct. Therefore $C$ contains at least $3m$ elements.
\end{proof}
\end{corollary}

Corollary~\ref{cor:s-geq-3k} implies that the smallest cap whose exclude set contains a $2$-point is a $6$-cap. In Figure~\ref{fig:6-cap} we give an example of a $6$-cap in $\mathbb{Z}_2^4$ (denoted by green diamonds). Points marked with a  {\color{red} {2}} are $2$-points, because they can be written in exactly two different ways as the sum of three cap points.

\begin{figure}[htbp]
    \centering
    \includegraphics[width=0.3\textwidth]{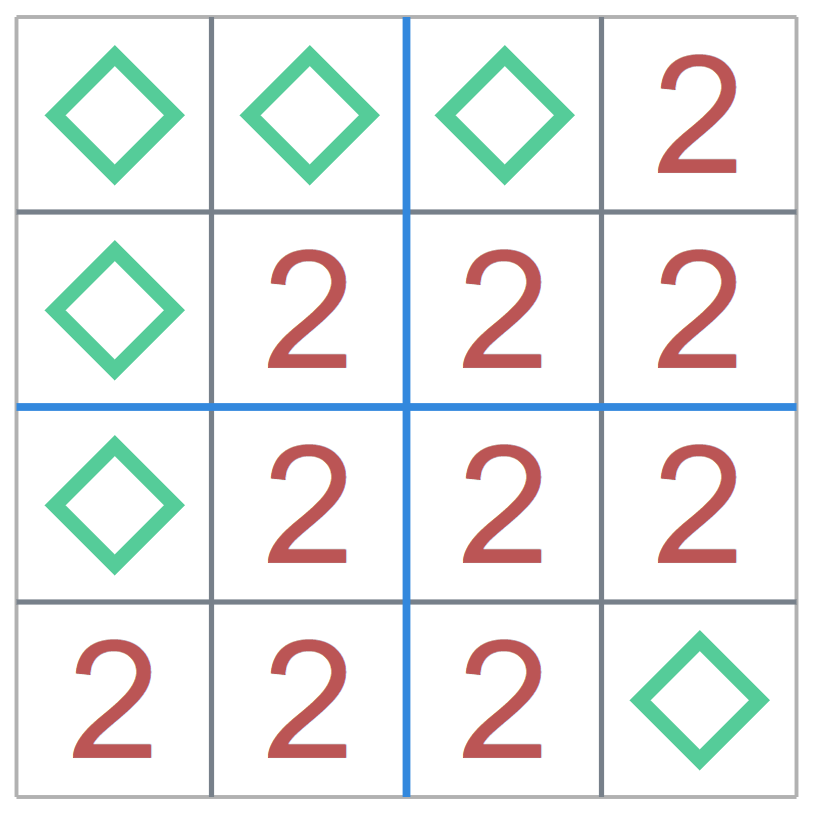}
    \caption{A 6-cap in $\mathbb{Z}_2^4$. Points with a red {\color{red}{2}} are $2$-points.}
    \label{fig:6-cap}
\end{figure}

\subsection{The Qap Visualizer}

A very useful tool for analyzing caps is the \textbf{Qap Visualizer} \cite{QapVis}, which was based on the Cap Builder for SET \cite{SETapp}. The term \deff{Qap} was used to distinguish from the standard caps in SET.

In this web-based app, a rectangular grid of squares represents the elements of $\mathbb{Z}_2^n$, in a square for even $n$ and in a half-square for odd $n$. The user selects one or more squares to build a cap, which are marked with green diamonds. With each new chosen square, the program determines which elements are in the exclude set and marks each with a red number denoting the multiplicity of the element. If the user hovers over a number, the triples that exclude this point are highlighted.
The user cannot select an element marked with a number, but choosing any empty square will create a larger cap. The user can also deselect an element by clicking on it again. 

We give an example in Figure~\ref{fig:qap-finder}. 
Observe that the quad closure of the cap consists of the green diamonds that denote cap points, together with the red numbers that denote the exclude points.

\begin{figure}[htbp]
    \centering
    \includegraphics[width=10cm]{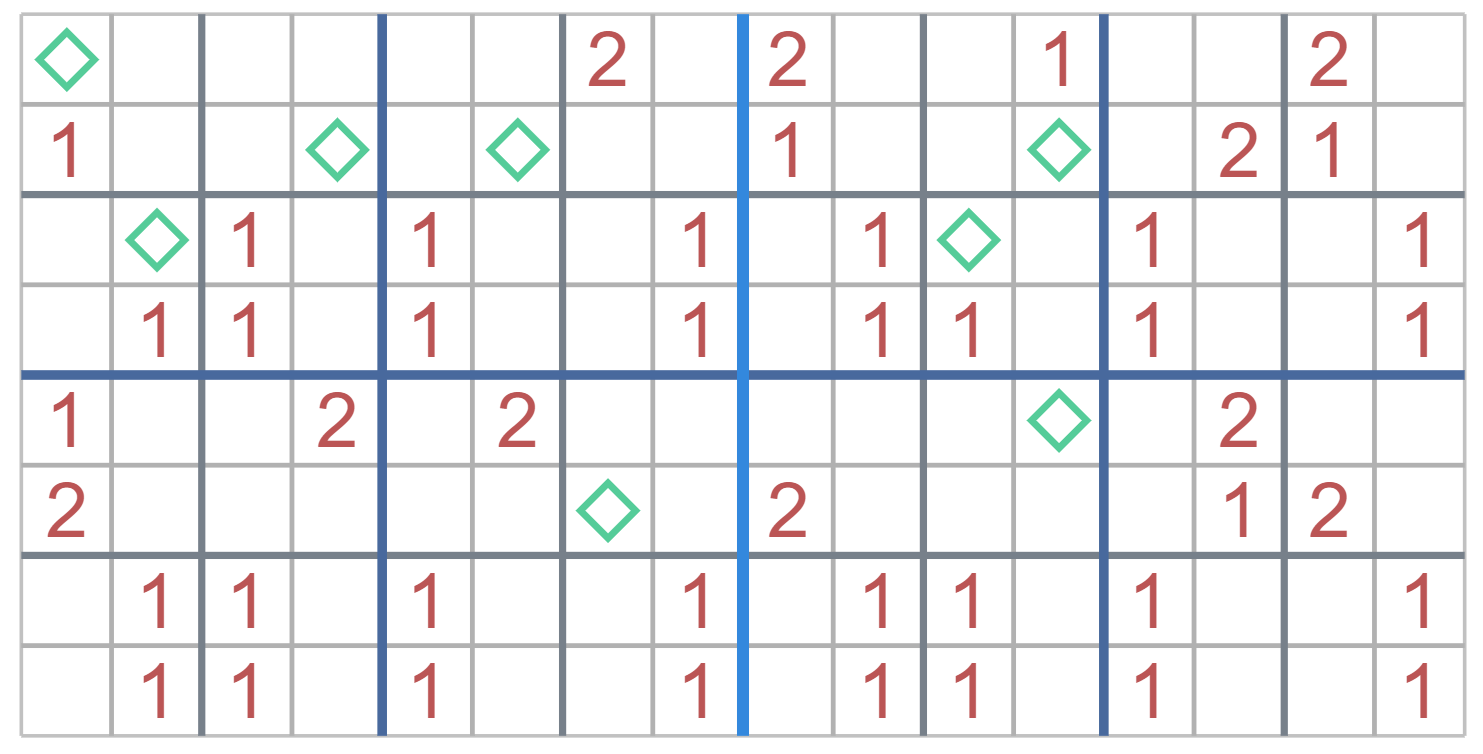}
    \caption{An $8$-cap and its exclude set in $\mathbb{Z}_2^7$, in the Qap Visualizer.}
    \label{fig:qap-finder}
\end{figure}

\begin{remark} 
\begin{enumerate}
    \item  The quad closure of a subset need not be closed under finding quads. In Figure~\ref{fig:bad-qc}, the elements
    \[
        (11,00,00), (11,00,11), (01,11,00), (01,11,11) \in \mathbb{Z}_2^6,
    \]
   circled in yellow, form a quad. The first three elements are in the exclude set and hence also the quad closure, but the fourth is outside the quad closure. However, we prove in Proposition~\ref{prop:repeated} later in this paper that repeatedly applying the quad closure to a set does eventually yield a set closed under taking quads.
    \item Maximal caps are complete, but there exist complete caps that are not maximal. For example, there exist complete caps of sizes 8 and 9 in $\mathbb{Z}_2^6$, but only the cap of size 9 is maximal. These are shown in Figure~\ref{fig:non-max-complete-caps}.
\end{enumerate}
\end{remark}

\begin{figure}[htbp]
    \centering
    \includegraphics[width=6cm]{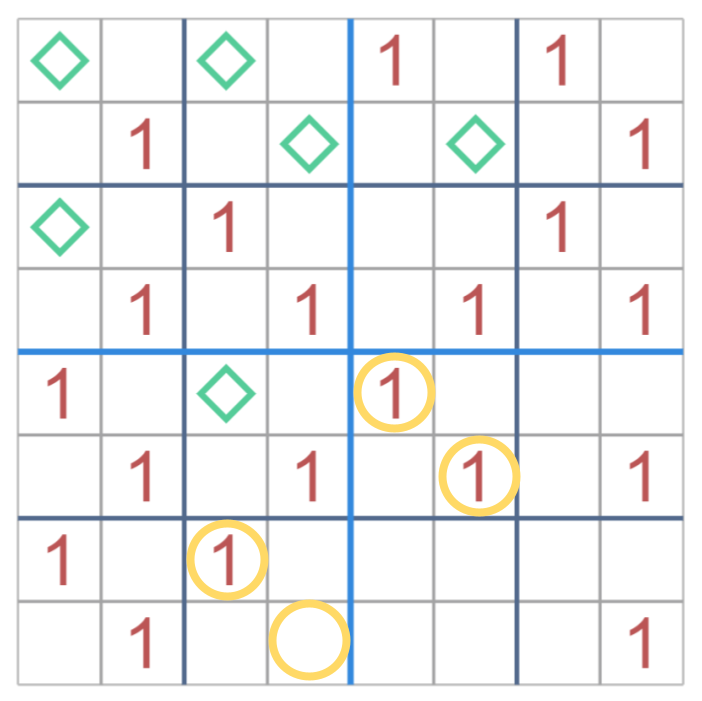}
    \caption{A $6$-cap in $\mathbb{Z}_2^6$ where the quad closure is not quad-closed.}
    \label{fig:bad-qc}
\end{figure}

\begin{figure}[htbp]
    \centering
    \includegraphics[width=0.45\textwidth]{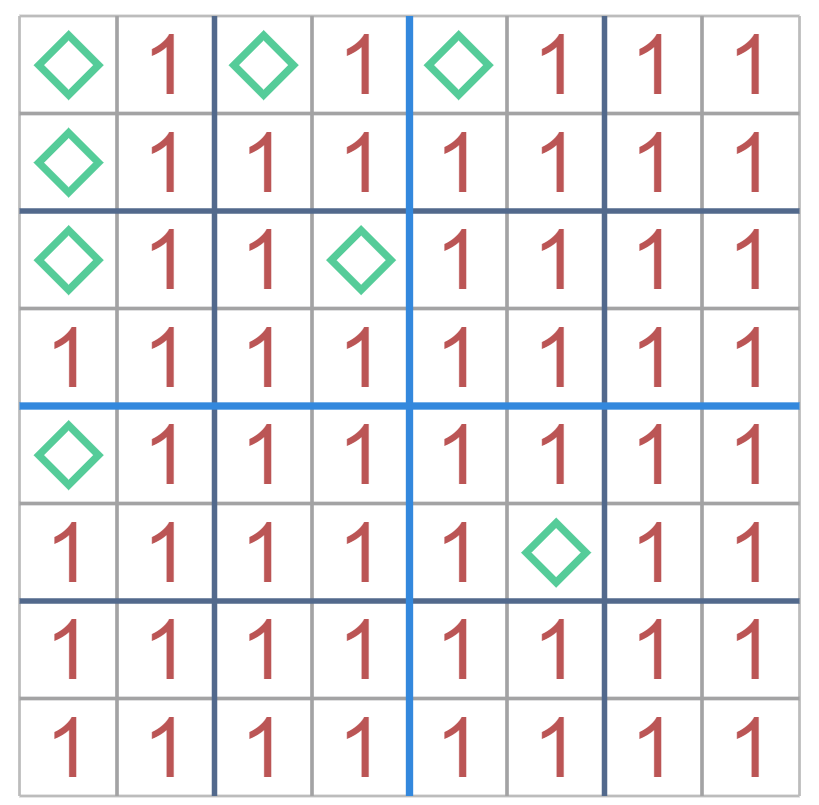}
    \includegraphics[width=0.45\textwidth]{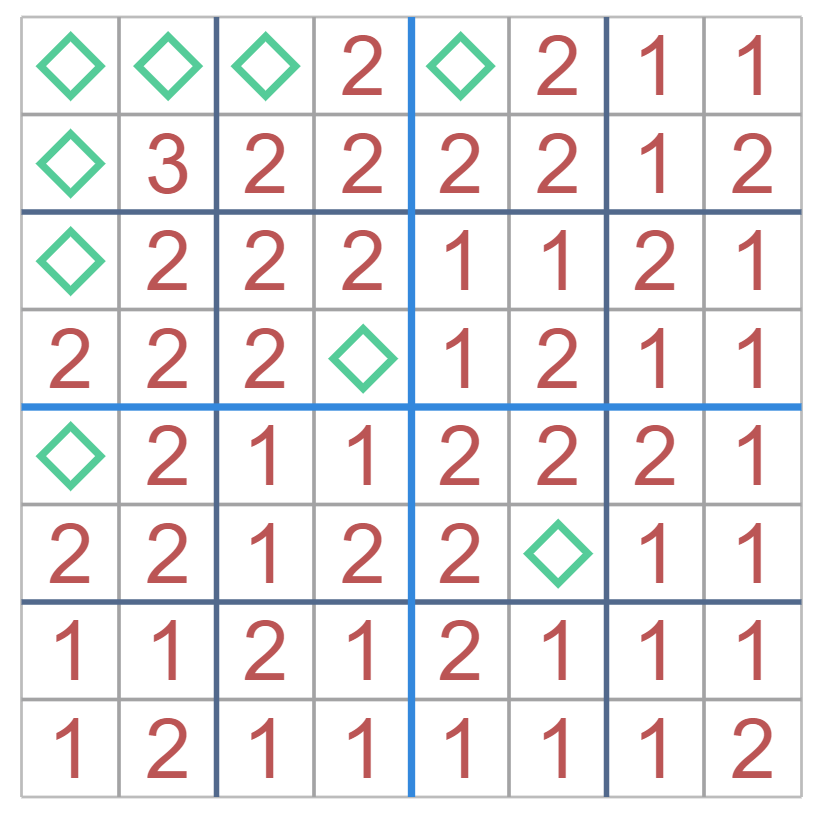}
    \caption{Complete caps of sizes 8 and 9 in $\mathbb{Z}_2^6$.}
    \label{fig:non-max-complete-caps}
\end{figure}

\section{Affine Geometry}
\label{sec:affinegeo}

Although we can study $AG(n,2)$ using the axioms of finite geometry, we find it more useful to work in $\mathbb{Z}_2^n$, which allows us to view affine lines, planes, etc.~as translates (cosets) of linear subspaces of a vector space. In particular, we can adapt well known results in linear algebra to study caps and flats in $AG(n,2)$ viewed as subsets of $\mathbb{Z}_2^n$.

We first recall some basic definitions and results from affine geometry. A detailed treatment can be found in \cite{audin}, among others. We close the section by highlighting specific properties of the finite affine geometry $AG(n,2)$.

\subsection{Preliminary definitions}
We recall basic definitions about affine linear algebra over an arbitrary field $K$. The main takeaways are that (1) any \emph{flat} (translate of a linear subspace) can be partitioned into subflats of a given size, and (2) affine transformations preserve affine combinations.

\begin{definition}
Let $V$ be a finite-dimensional vector space over a field $K$, and let $S = \{x_1, \ldots, x_n \} \subseteq V$.

\begin{enumerate}
    \item An \deff{affine combination} of $S$ is a linear combination
    \[
        \alpha_1 x_1 + \cdots + \alpha_n x_n
    \]
    where $\alpha_1, \ldots, \alpha_n \in K$ satisfy $\alpha_1 + \cdots + \alpha_n = 1$.
    
    \item An \deff{affine dependence} of $S$ is a linear combination
    \[
        \alpha_1 x_1 + \cdots + \alpha_n x_n = \vec{0},
    \]
    where $\alpha_1 + \cdots \alpha_n = 0$ and $\alpha_1, \ldots, \alpha_n$ are not all zero.

    \item The \deff{affine span} of a subset $S \subseteq V$ is the set $\aff(S)$ of all affine combinations of finitely many elements of $S$.

    \item $S$ is \deff{affinely dependent} if some element of $S$ is in the affine span of the other elements. Otherwise $S$ is \deff{affinely independent}.

    \item An \deff{$r$-dimensional affine subspace} $F$ of $V$, called an \deff{$r$-flat}, is defined to be the affine span of $r+1$ affinely independent elements of $V$, or equivalently, the translate $L+v$ of an \deff{$r$-dimensional linear subspace} $L$ of $V$. 

    \item An \deff{affine basis} for an $r$-flat $F \subseteq V$ is a set of affinely independent elements of $V$ whose affine span is $F$. Equivalently, if $F = L + v$, then an affine basis for $F$ is given by $\{ x + v \mid x \in B \cup \vec{0}\}$, where $B$ is a linear basis for $L$.

    \item An \deff{affine transformation} between vector spaces $V$ and $W$ over $K$ is a function $A : V \to W$ of the form $A(x) = M(x) + y$ for all $x \in V$, where $M : V \to W$ is a linear transformation and $y \in W$ is a fixed vector. 

    \item An affine transformation $A$ is an \deff{affine isomorphism}, or \deff{affine equivalence}, if $A$ is invertible. In this case, we say that $V$ and $W$ are \deff{affinely equivalent}, denoted $V \cong W$. 

    \item Subsets $S, T \subseteq V$ are \deff{affinely equivalent}, denoted $S \cong T$, if $A(S) = T$ for some affine isomorphism $A: V \to V$.

\end{enumerate}
\end{definition}

\begin{remark} The following results are standard and follow from the definitions above.

\begin{enumerate}

    \item Any affine basis for an $r$-flat contains $r+1$ elements, and every element of a flat can be written uniquely as an affine combination of affine basis elements.
    
    \item Two affinely independent sets of the same size are affinely equivalent.

    \item An affine $r$-flat $F = L + v$ is a coset of $L$, and $v$ can be any element of $F$. Thus, a partition of $V$ into the cosets of $L$ is also a partition into $r$-flats, where $L$ is the r-flat containing $\vec{0}$.
    
    \item When $F$ is a field with $q$ elements, an $r$-flat will contain $q^r$ elements.

    \item An affine transformation $A : V \to W$ \textbf{preserves affine combinations} in the following sense. For any $x_1, \ldots, x_n \in V$ and affine combination $\alpha_1 x_1 + \cdots + \alpha_n x_n$, where $\alpha_1, \ldots, \alpha_n \in F$ sum to $1$, we have
    \[
        A(\alpha_1 x_1 + \cdots + \alpha_n x_n)
        = \alpha_1 A(x_1) + \cdots + \alpha_n A(x_n).
    \]
    Consequently, affine transformations preserve affine independence and dependence.

\end{enumerate}
\end{remark}

Because we will rarely use any kind of span other than an affine one, we will drop the affine description. So if $\aff(S) = F$, we will say that \deff{$S$ spans $F$}. 

\subsection{Properties of Flats in AG(n,2)}

Affine linear algebra over the field $\mathbb{Z}_2$ has a particularly nice structure. Although several of the results below are true over any finite field, most make use of the fact that in vector spaces over $\mathbb{Z}_2$, the only scalars are 0 and 1. 

\begin{lemma} \label{rem:odd-sums}
Let $S = \{x_1, \ldots, x_m\} \subseteq \mathbb{Z}_2^n$. 
\begin{enumerate}
    \item An affine combination of $S$ is a sum of an odd number of elements of $S$.
    \item $S$ is affinely dependent if and only if a sum of an even number of elements of $S$ equals $\vec{0}$.
    \item Affine transformations of $\mathbb{Z}_2^n$ preserve addition of an odd number of elements.
    \item An $r$-flat $F$ containing a $t$-flat $F'$ can be partitioned uniquely into $2^{r-t}$ pairwise disjoint $t$-flats such that one of them is $F'$.
    \item Let $x$ and $y$ be points in an $r$-flat $F = L+v$, where $L$ is a linear subspace of dimension $r$. Then $x+y \in L$. 
\end{enumerate}

\begin{proof} \leavevmode
\begin{enumerate}
    \item Since the coefficients of an affine combination lie in $\{0,1\}$, the sum of the coefficients equals $1$ if and only if the number of elements is odd.
    \item A dependence means that some $x_i \in S$ is an affine combination of other elements of $S$, hence the sum of an odd number of them. Without loss of generality, assume $x_1 = x_2 + \ldots + x_{2t}$. Moving $x_1$ to the other side, and since $x_1 = - x_1$, we get that $x_1 + x_2 + \ldots + x_{2t} = \vec{0}$. 
    \item This follows directly from the fact that affine transformations preserve affine combinations.
    \item Let $F$ be an $r$-flat containing a $t$-flat $F'$, and let $v \in F'$. Then $F = L + v$ for some $r$-dimensional linear subspace $L$, and since $v = -v$ we can write $L = F+v$. Then $F' + v = L'$ is a linear subspace of $L$ with dimension $t$, and $F' = L' + v$. Since $\abs{L} = 2^r$ and $\abs{L'} = 2^t$, we can partition $L$ into $2^{r-t}$ cosets with respect to $L'$, say $ C_1, \ldots C_{r-t}$, one of which is $L'$. If we translate each of these cosets by $v$, the resulting sets $ C_1+v, \ldots C_{r-t}+v$ are $t$-flats partitioning $F$, and one of these is $L' + v = F'$.
    \item Let $x, y \in F = L+v$. Then $x = a + v$ and $y = b + v$ for some $a,b \in L$. Hence $x + y = (a + v) + (b + v) = a + b \in L$, since $v + v = \vec{0}$ and $L$ is closed under addition.
\end{enumerate}
\end{proof}
\end{lemma}

\begin{example}
We will find it useful to partition an $r$-flat into $4$-flats, especially because this is how they are displayed in the Qap Visualizer. For example, in Figure~\ref{fig:coords-example} the grid gives a partition of $\mathbb{Z}_2^6$ into four 4-flats.
\end{example}

\section{Classifying Caps}
\label{sec:geomprops}

In this section, we use tools from affine linear algebra to study quads and caps. We then use this to characterize small caps of at most seven points in terms of their geometric, algebraic, and combinatorial properties. In the next section, we will introduce new techniques to characterize caps of size greater than seven. 

Our primary goal is to determine the structure of $k$-caps in dimension $n$. This includes the following:

\begin{enumerate}
    \item Determine the dependence relations among elements in a $k$-cap.
    \item Determine the smallest dimension of a flat that admits a $k$-cap.
    \item Determine the number of affine equivalence classes of $k$-caps in each dimension.
    \item Determine the maximal size of a cap in an $r$-flat.
\end{enumerate}

\bigskip

The following theorem is a compilation of our main findings about caps of sizes up to 9. All of these results are described and proved thoroughly in subsequent sections.

\begin{theorem}[Summary of Cap Characterizations]\leavevmode
\begin{enumerate}
    \item Caps of size $k \leq 4$ are affinely independent, and maximal in dimension $k-1$.
    \item All $5$-caps are affinely independent.
    \item A $6$-cap is either affinely independent, or has a $6$-point dependence relation. In the latter case it is maximal in dimension $4$.
    \item A $7$-cap is either affinely independent or has a $6$-point dependence relation. In the latter case it is maximal in dimension $5$.
    \item An $8$-cap is either affinely independent, contains a $6$-point dependence relation, or contains an $8$-point dependence. In the third case it is complete, but not maximal, in dimension $6$. 
    \item A $9$-cap is either affinely independent, has an $8$-point dependence relation, or has two $6$-point dependence relations. In the last case it is maximal in dimension $6$.
\end{enumerate}
\end{theorem}
   
\begin{remark}
    For each $k$, each different case represents a single affine equivalence class of $k$-caps, so that when $k \leq 5$ there is only one equivalence class, for $k=6$  and $k=7$ there are two classes each, for $k=8$ there are three classes, and for $k=9$ there are four classes. 
  
    In what follows, we also determine the \emph{dimension} of each equivalence class, i.e.~the smallest dimension that can contain a cap in that equivalence class. For example, of the three types of equivalence classes of $8$-caps, two types can occur in dimension $6$, while the other type requires dimension $7$.
\end{remark}
  
It is of great general interest to determine $M(r)$, the size of a maximal cap in each dimension $r$. The results listed above and proved in this paper show that in dimensions 1--6, the values of $M(r)$ are $2, 3, 4, 6, 7, 9$, respectively. We have also computed the following values, proofs of which will appear in a forthcoming paper. (Here $r_k$ denotes the smallest possible dimension of a $k$-cap.)
\begin{itemize}
      \item $M(7) = 12$ and $M(8) = 18$.
      \item $r_{11} = r_{12} = 7$, 
      $r_{13} = \cdots = r_{18} = 8$, and
     $r_{19} = \cdots = r_{23} = 9$.
\end{itemize}

\subsection{Geometric descriptions of quads and caps} \label{subsec-geom}

We can now describe properties of quads and caps in the context of affine geometry.

\begin{proposition} \label{thm:caps-are-planes}
Quads in $AG(n,2)$ are exactly the $2$-flats, or affine planes.

\begin{proof}
Let $\{a,b,c,d\}$ be a quad. Then $a+b+c+d=\vec{0}$, so $d = a+b+c$ is an affine combination of the three elements $a,b,c$. Since the only affine combinations of $a,b,c$ are sums of an odd number of these elements, it follows that $\aff \{a,b,c\} = \{a,b,c,d\}$, so the quad is a flat with $4 = 2^2$ elements, hence it is a $2$-flat.

Now let $F$ be a $2$-flat. Then $F = \aff \{x,y,z\}$ for some affinely independent elements $x,y,z$, and by item 1 of Lemma~\ref{rem:odd-sums} we know $\aff \{x,y,z\} = \{x,y,z,x+y+z\}$. Since $x,y,z$ are independent, we know that $x+y+z$ is distinct fourth element. Summing these elements yields
\[
    x + y + z + (x+y+z) = (x+x) + (y+y) + (z+z) = \vec{0},
\]
so $F = \{x,y,z,x+y+z\}$ is a quad.
\end{proof}
\end{proposition}

The following lemma will be used repeatedly in the remainder of this paper.

\begin{lemma} \label{lem:5-or-more}
Let $C$ be a cap, and suppose $y \in C$ is also in the span of $C - \{y\}$. Then $y$ is the sum of $2m+1$ elements of $C' = C - \{y\}$ where $2m+1 \geq 5$.

\begin{proof}
Suppose $y$ is an affine combination of $x_1, \ldots, x_n \in C'$. Recall that affine combinations of $x_1, \ldots, x_n$ are sums of an odd number of elements. Since $y \notin C'$, it cannot be a sum of one element of $C'$. If $y$ is the sum of three distinct elements of $C'$, say $y = x_i + x_j + x_k$, then we have $x_i + x_j + x_k + y = \vec{0}$, so $\{x_i, x_j, x_k,y\}$ is a quad, contradicting the fact that $C$ is a cap. Hence $y$ must be the sum five or more elements of $C'$.
\end{proof}
\end{lemma}

Now we show that affine independence is a sufficient condition to be a cap. 

\begin{proposition} \label{thm:ind-implies-cap}
Let $S$ be an affinely independent subset of $\mathbb{Z}_2^n$. Then $S$ is a cap, and every element of $\exc(S)$ is a $1$-point.

\begin{proof}
Suppose that $S$ is affinely independent but not a cap. Then it contains a quad, so there are distinct elements $a,b,c,d \in S$ such that $a+b+c+d=\vec{0}$. Then $d = a+b+c$ is an affine combination of elements of $S$, contradicting the assumption that $S$ is affinely independent. Therefore $S$ must be a cap.

For the rest, let $p \in \exc(S)$ be an $m$-point and suppose that $m>1$. Then there are at least two distinct triples $\{x_1,y_1,z_1\}$ and $\{x_2,y_2,z_2\}$ in $S$ such that $x_1 + y_1 + z_1 = p$ and $x_2 + y_2 + z_2 = p$. Then $x_1 + y_1 + z_1 = x_2 + y_2 + z_2$, and adding $x_2 + y_2$ to both sides yields
\[
    x_1 + y_1 + z_1 + x_2 + y_2 = z_2.
\]
By Corollary~\ref{disjoint-triples}, the two sets of triples are disjoint, so we have five distinct elements of $S$ whose sum is an element of $S$. This contradicts the assumption that $S$ is affinely independent, so it must be that $m = 1$.
\end{proof}
\end{proposition}

To see that affine independence is not a necessary condition to be a cap, we have the following example.

\begin{example}
Consider the $6$-cap $C \subseteq \mathbb{Z}_2^5$ in Figure~\ref{fig:dep-cap}. Note that the exclude set contains $2$-points, so by Proposition~\ref{thm:ind-implies-cap}  $C$ is affinely dependent.
\end{example}

Since every quad completion of an affinely independent cap is a $1$-point, there is a unique triple in the cap that has the point as its sum. Hence, we have the following corollary to Proposition~\ref{thm:ind-implies-cap}.

\begin{corollary} \label{cor:ind-cap-exclude-size}
Let $C$ be an affinely independent $k$-cap. Then $\lvert \exc(C) \rvert = \binom{k}{3}$.
\end{corollary}

\begin{figure}[htbp]
    \centering
    \includegraphics[width=6cm]{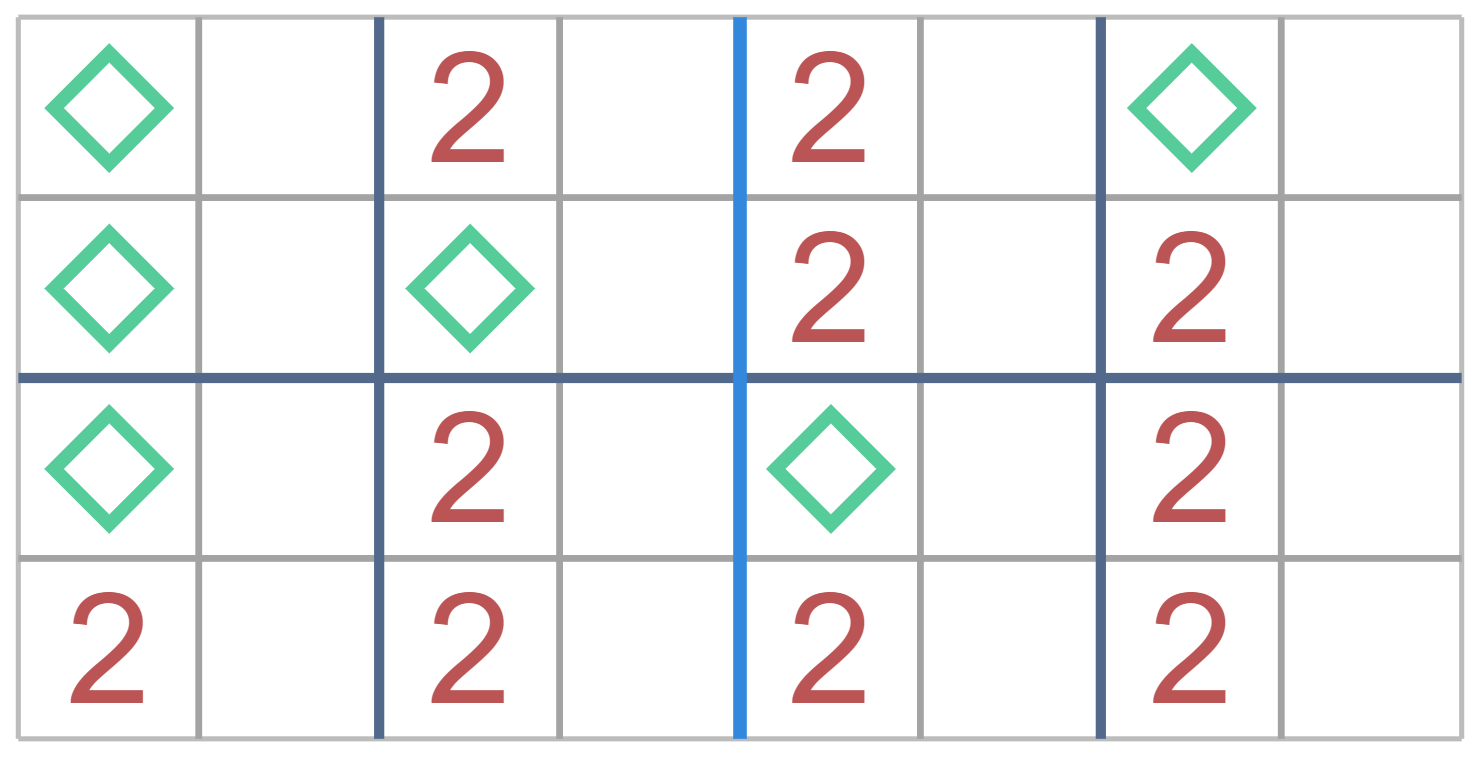}
    \caption{An affinely dependent $6$-cap in $\mathbb{Z}_2^5$.}
    \label{fig:dep-cap}
\end{figure}

Our next result shows that quad closure of a set $S$ (not necessarily a cap) is contained in the affine span of $S$, and by repeatedly applying the quad closure we will eventually get all of $\aff(S)$.

\begin{proposition} \label{prop:repeated}
Let $S \subseteq \mathbb{Z}_2^n$ and let $\qc^i(S)$ denote the result of taking the quad closure of $i$ times. Then: 
\begin{enumerate}
    \item $\qc(S) \subseteq \aff(S)$.
    \item For some $i$, $\qc^i(S) = \aff(S)$.
    \item $S = \qc(S)$ if and only if $S = \aff(S)$.
\end{enumerate}

\begin{proof} \leavevmode
\begin{enumerate}
    \item Recall that the exclude set $\exc(S)$ consists of all sums of $3$ distinct elements of $S$. Since the affine span $\aff(S)$ contains of all sums of an odd number of elements of $S$, it certainly contains sums of $1$ and $3$ elements. Thus,
    \[
        \qc(S) = S \cup \exc(S) \subseteq \aff(S),
    \]
    
   \item If $x \in S$, we will say $x$ is in the $0$th quad closure of $S$, and we know from above that $\qc^1(S)$ consists of all sums of $1$ and $3$ points in $S$. Suppose that all sums of $2m + 1$ points in $S$ lie in $\qc^i(S)$, for some fixed $i$, and let $x = s_1+\ldots+s_k$ be the sum of $k = 2(m + 1) + 1 = 2m + 3$ points in $S$. Then $k-2 = 2m + 1$, so $r = s_1+\ldots+s_{k-2}$ is in $\qc^i(S)$. This means that $x = r + s_{k-1} + s_k$ must be in $\qc^{i+1}(S)$. Thus, all sums of $2(m + 1) + 1$ points in $S$ lie in $\qc^i(S)$. By induction, each point in the affine span of $S$ is contained in $\qc^i(S)$ for some $i$. 

   \item Since $S \subseteq \qc(S) \subseteq$ aff(S), then $S = \aff(S)$ implies $S=\qc(S)$.  

   Conversely, if $S = \qc(S)$ then applying the quad closure again we obtain  $\qc(S) = \qc^2(S)$, so  $S = \qc^2(S)$. Continuing, we see that for some $i$ we have $S = \qc^i(S) = \aff(S)$.   

\end{enumerate}
\end{proof}
\end{proposition}

We now define the dimension of a cap, and introduce some terminology for maximal caps and the minimal dimension of a $k$-cap. (We introduced the definitions of $M(r)$ and $r_k$ at the beginning of this chapter, but recall them here for completion.)

\begin{definition}
Let $C$ be a cap in $\mathbb{Z}_2^n$ be a cap. 
\begin{enumerate}
    \item The \deff{dimension} of $C$, denoted $\dim(C)$, is the dimension of $\aff(C)$, the smallest flat containing $C$.
    \item We denote by $r_k$ the \textbf{smallest possible dimension} of a $k$-cap. 
    \item We denote by $M(r)$ the \textbf{maximal cap size} in an $r$-flat in $\mathbb{Z}_2^n$.
    \item If $C$ is $r$-dimensional and complete, we say that $C$ is a \deff{complete cap in dimension $r$}, and that $C$ \deff{completes} the $r$-flat $\aff(C)$.
\end{enumerate}
\end{definition}

\begin{proposition} \label{prop:complete-maximal-cap-size} \leavevmode
\begin{enumerate}
    \item All $k$-caps of dimension $k-1$ are affinely independent.
    \item For any $k$-cap $C$, we have $r_k \leq dim(C) \leq k-1$.
    \item The numbers $r_k$ are non-decreasing.
    \item If every $k$-cap of dimension $r$ is complete, then $M(r) = k$.
\end{enumerate}

\begin{proof} \leavevmode
\begin{enumerate}
    \item A $k$-cap of dimension $k-1$ spans a $(k-1)$-flat. The only way that the affine span of $k$ elements can be $k-1$ dimensional is if they are affinely independent. 
    \item By definition, $r_k$ is the minimum dimension of a flat that can contain a $k$-cap, hence $r_k \leq dim(C)$. Since the affine span of any $k$ points in $\mathbb{Z}_2^n$ has dimension at most $k-1$, it follows that the dimension of a $k$-cap is at most $k-1$.
    \item Since a subset of a cap is also a cap, it follows that the $r_k$ are non-decreasing.
    \item This is a restatement of Proposition~ \ref{basic-props}, which says that if all $k$-caps are complete, then they are maximal.
\end{enumerate}
\end{proof}
\end{proposition}

\subsection{Caps of sizes 1--5}

When $k \leq 5$, $k$-caps are easy to characterize, as they are all affinely independent. 

\begin{theorem} \label{thm:affine-ind}
Let $C$ be a $k$-cap. Then
\begin{enumerate}
    \item When $k \leq 4$, $C$ is a maximal cap in dimension $k-1$.
    \item When $k \leq 5$, $C$ is affinely independent.
    \item When $k \leq 5$, $r_k = k-1$.
    \item When $0 \leq r \leq 3$, $M(r) = r+1$.
\end{enumerate}

\begin{proof} \leavevmode

\begin{enumerate}

    \item Suppose $k \leq 4$. By (2), which we prove below, $C$ is affinely independent and the span of $C$ consists of all sums of one or three elements (since $k < 5$), so $\aff(C) = \qc(C)$, hence $C$ is a complete cap. Since all caps of size $k$ are complete in dimension $k-1$, by Proposition~\ref{prop:complete-maximal-cap-size} $C$ must be maximal and hence $M(k-1) = k$ .

    \item Suppose $k \leq 5$. The only affine combinations of elements of $C$ are sums of 1, 3, or 5 elements of $C$. Since there are at most five elements in $C$, the only way an element of $x \in C$ is a (non-trivial) affine combination of the other elements is if $x$ is the sum of 3 of them, which means that $x$ is a quad completion, contradicting the fact that $C$ is a cap. Therefore $C$ is affinely independent.

    \item All caps of size $k \leq 5$ are affinely independent, and hence have dimension $k-1$. Thus they cannot exist in any dimension smaller than $k-1$, so $r_k = k-1$ for $k \leq 5$.

    \item In the proof of (1), we showed that $M(k-1) = k$ for $1 \leq k \leq 4$. Setting $r = k-1$, we have $M(r) = r+1$ for $0 \leq r \leq 3$.

\end{enumerate}
\end{proof}
\end{theorem}

We can say a bit more about $5$-caps, namely that the quad closure is a $4$-flat minus a single point $y$. The results below provide a seamless transition to the study of $6$-caps.

\begin{theorem}[Properties of 5-caps] \label{thm:5-caps}
Let $C = \{x_1, x_2, x_3, x_4, x_5\}$ be a $5$-cap and $y =  x_1 + x_2 +x_3 + x_4 + x_5$ be the sum of the elements of $C$.
\begin{enumerate}
    \item The set $\qc(C) \cup \{y\}$ is a $4$-flat, and the union is disjoint.
    \item $C$ is not complete, but $C \cup \{y\}$ is a complete $6$-cap in dimension $4$.
    \item Every point in $\exc(C \cup \{y\})$ has multiplicity $2$.
\end{enumerate}

\begin{proof} \leavevmode
\begin{enumerate}
    \item By Remark~\ref{rem:odd-sums}, every element of $\aff(C)$ is the sum of an odd number of elements of $C$. Since $\lvert C \rvert = 5$, this means $\aff(C)$ consists of sums of one, three, and five distinct elements of $C$. The sums of one or three elements together with $C$ form the quad closure $\qc(C)$, and $y$ is the unique sum of five elements of $C$, so 
    \[
        \aff(C) = \qc(C) \cup \{y\}.
    \]

    We now show that $y \notin \qc(C)$, so this union is disjoint. Since $C$ is affinely independent, all sums of odd numbers of elements are distinct, so $y$ cannot be a sum of one or three elements of $C$. Hence, $y \notin \qc(C)$.

    \item Since there are $\binom{5}{1} + \binom{5}{3} + \binom{5}{5} = 16$ distinct sums of odd numbers of elements of $C$, $\aff(C)$ must be a $4$-flat. Since $\qc(C)$ contains all elements except $y$, the unique sum of five elements, this means that $C$ is not a complete cap in dimension $4$. But when we include $y$ we have $\qc(C \cup \{y\}) = \aff(C)$, so $C \cup \{y\}$ is a complete $6$-cap in dimension $4$.

    \item Now let $C' = C \cup \{y\}$, and note that since $y = x_1 + x_2 +x_3 + x_4 + x_5$, the sum of all six elements of $C'$ is $\vec{0}$. Rename the elements of $C'$ as $\{a,b,c,d,e,f\}$. Let $p \in \exc(C')$, and suppose without loss of generality that $p = a+b+C$. Since $a+b+c+d+e+f = \vec{0}$, it follows that $a+b+c=d+e+f$. Therefore $p$ has multiplicity $2$ or higher.
    By Corollary~\ref{cor:s-geq-3k}, a $6$-cap is too small to have exclude points of multiplicity $3$ or higher, so each $p$ must be a $2$-point. 

\end{enumerate}
\end{proof}
\end{theorem}

\begin{example}
In Figure~\ref{fig:5-cap} we give an example of a $5$-cap in a $4$-flat. Note that quad closure of the cap contains all but one point.

\begin{figure}[htbp]
    \centering
    \includegraphics[width=0.3\textwidth]{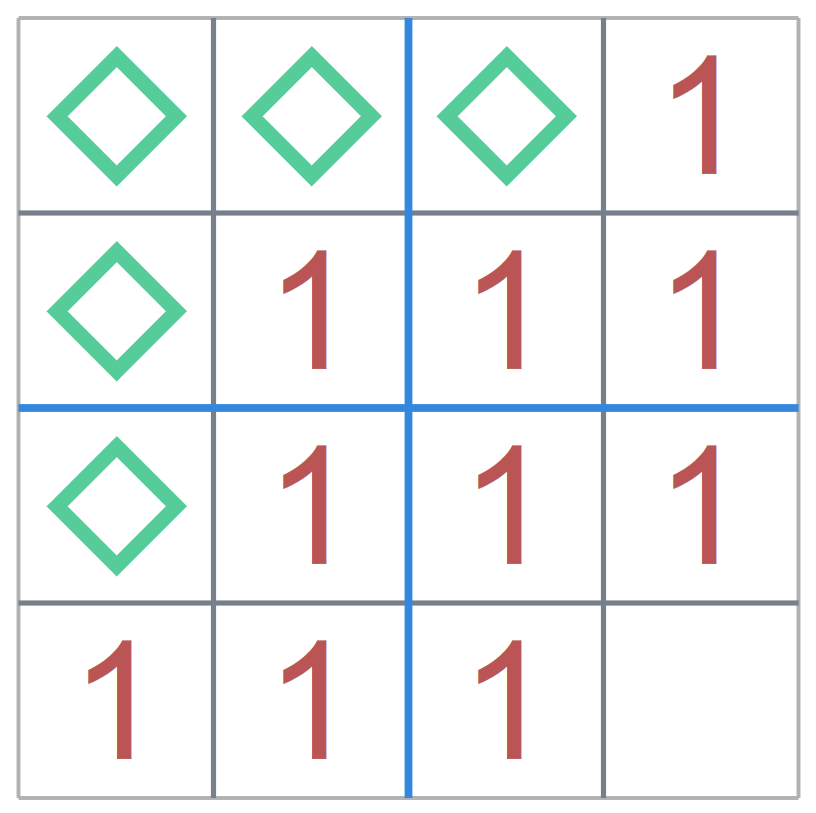}
    \caption{A $5$-cap in $\mathbb{Z}_2^4$.}
    \label{fig:5-cap}
\end{figure}
\end{example}

\subsection{6-caps}

When $k \geq 6$, $k$-caps can be affinely dependent, so we will analyze them by their cap dimension. It turns out that understanding $6$-caps is key to understanding caps in general.

There are two types of $6$-caps, one affinely independent and the other with a single dependence relation involving all 6 cap points. Moreover, caps of the latter type are maximal. 

\begin{theorem}[Characterization of 6-caps] 
\label{thm:6-cap-in-4-flat}
Let $C$ be a $6$-cap. Then
\begin{enumerate}
    \item The dimension of $C$ is either $4$ or $5$.
    \item If $\dim(C) = 5$, then $C$ is affinely independent.
    \item If $\dim(C) = 4$, then $C$ is a maximal cap in dimension $4$, and every $p \in \exc(C)$ is a $2$-point.
\end{enumerate}

\begin{proof} \leavevmode
\begin{enumerate}

    \item By Proposition~\ref{prop:complete-maximal-cap-size}, $r_6 \leq \dim(C) \leq 5$. Since $r_5 = 4$, it follows that $r_6 \geq 4$. By Theorem~\ref{thm:5-caps}, we know there are $6$-caps in dimension $4$, so $r_6 = 4$. Hence $\dim(C)$ equals $4$ or $5$.

    \item Suppose $\dim(C) = 5$. By statement 1 of Proposition~\ref{prop:complete-maximal-cap-size}, we know that $C$ is affinely independent.

    \item Suppose $\dim(C) = 4$. Let $F$ be the $4$-flat spanned by $C$, let $C'$ be a subset of $C$ with five elements, and let $y \in C - C'$ be the unique remaining element. Observe that $C'$ is a $5$-cap contained in the $4$-flat $F$. By part 2 of Theorem~\ref{thm:5-caps}, we know that $F$ is the union of $\qc(C')$ and one additional point, which must be $y$, and that $C = C' \cup \{y\}$ completes $F$. By part 3 of the theorem, every point in $\exc(C) = \exc(C' \cup \{y\})$ is a $2$-point.
    
    Since the argument above proves that every $6$-cap in dimension $4$ is complete, it follows from Proposition~\ref{prop:complete-maximal-cap-size} that all such caps are maximal.
    
\end{enumerate}
\end{proof}
\end{theorem}

\begin{example}
In Figure~\ref{fig:6-cap} we give an example of a $6$-cap in a $4$-flat. Note that this is the $5$-cap from Figure~\ref{fig:5-cap} with its one empty spot filled, which turns the $1$-points into $2$-points.
\end{example}

We can now compute $M(4)$, the maximal cap size in dimension $4$, as well as $r_6$ and $r_7$, the minimal dimensions of $6$-caps and $7$-caps. 

\begin{corollary} \label{cor:6-7-cap-dim}
$M(4) = 6$, $r_6 = 4$, and $r_7 = 5$.

\begin{proof}
By statement 3 of Theorem~\ref{thm:6-cap-in-4-flat}, caps of size $6$ are maximal in dimension $4$, so $M(4) = 6$. The value $r_6 = 4$ is determined in the proof of statement 1. Since $M(4) = 6$, a $7$-cap cannot exist in dimension $4$, so $r_7$ must be at least $5$. Indeed, a $7$-cap in dimension $5$ can be constructed from any $6$-cap $C$ in dimension $4$ by adding any point outside of the $4$-flat containing $C$, since $\qc(C)$ is contained in this $4$-flat.
\end{proof}
\end{corollary}

In Figures~\ref{fig:6-cap} and~\ref{fig:7-cap-dim-5} we give examples of a $6$-cap of dimension $4$ and a $7$-cap in dimension $5$.

\begin{figure}[htbp]
    \centering
    \includegraphics[width=6cm]{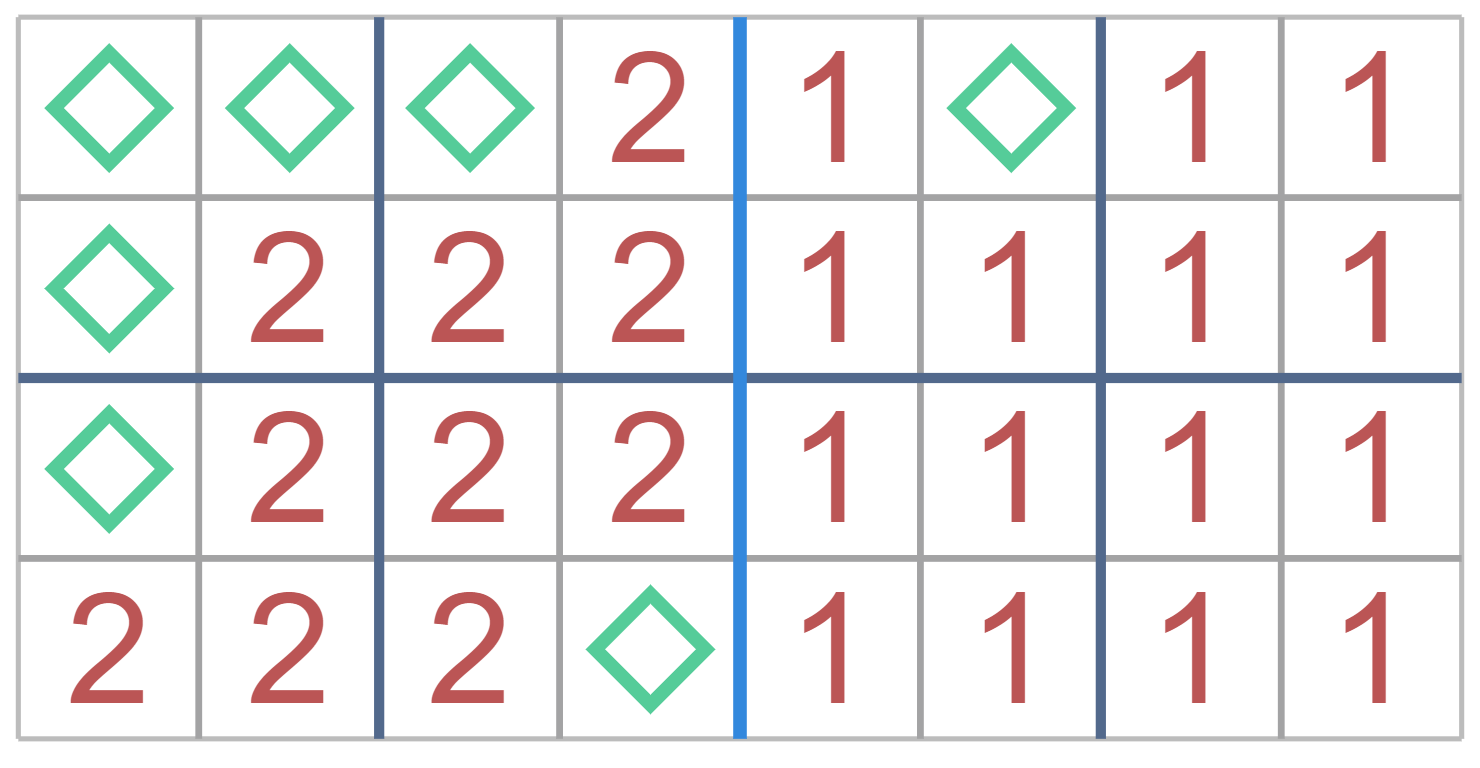}
    \caption{A $7$-cap in $\mathbb{Z}_2^5$.}
    \label{fig:7-cap-dim-5}
\end{figure}


The remainder of the results in this subsection will be used to classify caps of sizes $7$, $8$ and $9$. 

The following corollary to Theorem~\ref{thm:6-cap-in-4-flat} shows that any cap with an exclude point of multiplicity $2$ or higher must contain a $6$-cap of dimension $4$ as a subset. 

\begin{corollary} \label{cor:mult-2-or-higher}
Let $C$ be a cap. Then $C$ has an exclude point of multiplicity $2$ or higher if and only if $C$ contains a $6$-cap of dimension $4$.

\begin{proof} \leavevmode
\begin{description}
    \item[$(\Rightarrow)$] Suppose $C$ contains a subset $C'$ of $6$ points that span in a $4$-flat. By Theorem~\ref{thm:6-cap-in-4-flat} every point of $\exc(C')$ has multiplicity at least $2$. Since $\exc(C') \subseteq \exc(C)$ and multiplicities can only increase in going from $C'$ to $C$, $\exc(C)$ must contain a point with multiplicity at least $2$.
    
    \item[$(\Leftarrow)$] Suppose $\exc(C)$ contains a point $p$ with multiplicity $2$ or higher. Then there are distinct points $a,b,c,d,e,f \in C$ such that $a+b+c = d+e+f = p$. Then adding $b+c$ to both triple sums yields
    \[
        a = b+c+d+e+f.
    \]
    Since this is a sum of $5$ elements, $a$ is an affine combination of these five elements. Therefore $\{a,b,c,d,e,f\}$ is an affinely dependent $6$-cap. By Theorem~\ref{thm:6-cap-in-4-flat} an affinely dependent $6$-cap must have dimension $4$.
\end{description}
\end{proof}
\end{corollary}

Let $C$ be a cap containing a $6$-cap that spans a $4$-flat $F$. Then any translation of $F$ contains at most one point from $C$. Before proving this, we give an example.
 
\begin{example}
In Figure~\ref{fig:4-flat-partition}, the $4$-flat in the upper-left corner contains 6 cap points, and we see that no more than one of the remaining cap points lies in any other $4$-flat.

\begin{figure}[htbp]
    \centering
    \includegraphics[width=6cm]{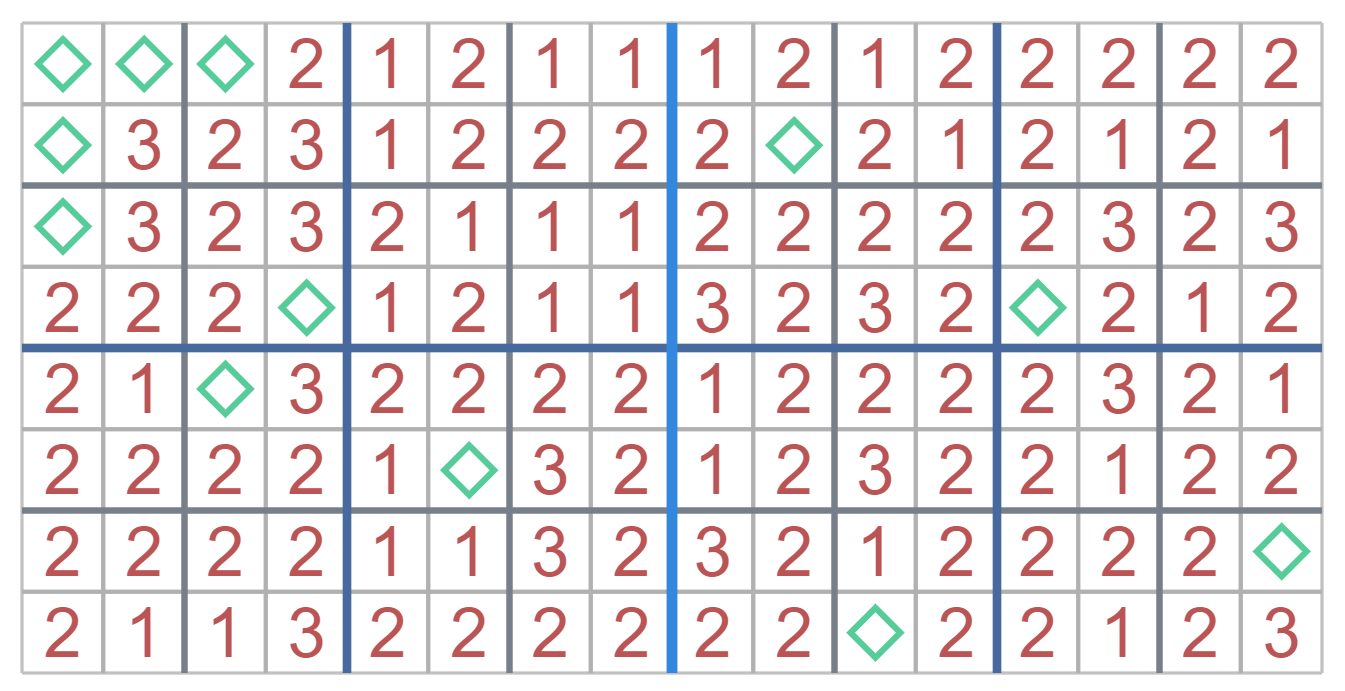}
    \caption{A cap containing six points in a $4$-flat.}
    \label{fig:4-flat-partition}
\end{figure}
\end{example}

\begin{proposition} \label{prop:1-per-4-flat}
Let $C$ be a cap containing a $6$-cap that spans a $4$-flat $F$.
Let $F_1, F_2, \ldots, F_{2^{n-4}}$ be a partition of $\mathbb{Z}_2^n$ into $4$-flats such that $F_1 = F$. Then each $4$-flat $F_i$ with $i \geq 1$ contains at most one point in $C$, i.e.,
\[
  \lvert F_i \cap C \rvert \leq 1 \qquad \text{for $i = 2, \ldots, 2^{n-4}$}.
\]

\begin{proof}
Let $x_1, \ldots, x_6 \in C$ be points contained in $F_1$, and let $y \in C$ be another point. By Corollary~\ref{cor:6-7-cap-dim}, a $4$-flat cannot contain more than $6$ cap points, so it follows that $y \in F_k$ for some $k \geq 2$.
 
Let $L$ be the linear subspace whose cosets are the flats $F_1, \ldots, F_{2^{n-4}}$. There are $\binom{6}{2} = 15$ excludes of $C$ of the form $p_{ij} = x_i + x_j + y$ for $i \leq j$. Since $x_i, x_j \in F_1$, by Lemma~ \ref{rem:odd-sums} we know that $x_i + x_j \in L$. Therefore
\[
    p_{ij}
    \in (x_i + x_j + y) + L
    = [(x_i + x_j) + L)] + (y + L)
    = L + (y + L)
    = y + L
\]
so $p_{ij} \in y + L = F_k$. Also note that $p_{ij} \neq y$, or else $p_{ij} = x_i + x_j + y$ would imply that $x_i = x_j$, which contradicts the fact that these points are distinct.

Now by Corollary~\ref{disjoint-triples}, if $p_{ij}=p_{\ell m}$, then $\{x_i, x_j, y\}$ and $\{x_{\ell}, x_m, y\}$ must be disjoint, which is impossible since $y$ is in both sets. Since the $15$ elements $p_{ij}$ are distinct points of the $4$-flat $F_k$, and there are only $16$ points in $F_k$, we conclude that $y$ is the only possible point of $C$ in $F_k$.
\end{proof}
\end{proposition}

\subsection{7-caps}

We can now characterize $7$-caps. Since $r_7 = 5$, any $7$-cap must have dimension $5$ or $6$.

\begin{theorem}[Characterization of 7-caps] 
\label{thm:7-cap-in-5-flat}
Let $C$ be a $7$-cap. Then
\begin{enumerate}
    \item $\dim(C)$ is either $5$ or $6$.
    \item If $\dim(C) = 6$, then $C$ is affinely independent.
    \item If $\dim(C) = 5$, then $C$ is maximal in dimension $5$ and $\exc(C)$ contains a $2$-point.
\end{enumerate}

\begin{proof}
\begin{enumerate}

    \item By Proposition~\ref{prop:complete-maximal-cap-size}, $r_7 \leq \dim(C) \leq 6$. By Corollary~\ref{cor:6-7-cap-dim} we know that $r_7 = 5$, so $\dim(C)$ equals $5$ or $6$.

    \item Suppose $\dim(C) = 6$. By statement 1 of Proposition~\ref{prop:complete-maximal-cap-size}, we know that $C$ is affinely independent.

    \item Suppose $C$ has dimenson $5$. Then $C$ is not affinely independent, but it contains a $6$-cap $C'$ of dimension $5$ which is affinely independent. Write $C' = \{x_1, \ldots, x_6\}$ and $C = \{x_1, \ldots, x_7\}$. Then $x_7$ must be an affine combination of points of $C'$. It cannot be the sum of $3$ points because $C$ is a cap, and it cannot be the sum of $7$ or more points since $\abs{C'} = 6$. This means $x_7$ is the sum of $5$ affinely independent points in $C'$. Without loss of generality, assume $x_7 = x_1 + \cdots + x_5$. Then $C'' = \{x_1, \ldots, x_5, x_7\}$ is a $6$-cap that spans a $4$-flat.
    
    By statement 4 of Lemma~\ref{rem:odd-sums}, we can partition the $5$-flat $\aff(C)$ into two $4$-flats, $F_1$ and $F_2$, one of which, say $F_1$, is the span of $C''$. By Corollary \ref{prop:1-per-4-flat}, there can be at most one point of $C$ in $F_2$, namely $x_6$, which means that $C$ is complete in dimension $7$. 
    
    Since the argument above proves that every $7$-cap in dimension $5$ is complete, it follows from Proposition~\ref{prop:complete-maximal-cap-size} that all such caps are maximal.
    
    Finally, note that $x_7 = x_1 + \cdots + x_5$, and this is the only affine dependence of elements of $C$ involving $x_7$, because $C' = \{x_1, \ldots, x_6\}$ is affinely independent. Therefore the point $p = x_1 + x_2 + x_3 = x_4 + x_5 + x_7$ in $\exc(C)$ cannot be written in any other ways as the sum of three elements of $C$, so it is a $2$-point.


\end{enumerate}
\end{proof}
\end{theorem}

We can now determine the maximal cap size in dimension $5$, as well as the smallest dimension of a flat that can contain an $8$-cap or a $9$-cap.

\begin{corollary} \label{cor:max-cap-dim-5}
$M(5) = 7$, $r_8 = 6$, and $r_9 = 6$.

\begin{proof}
By statement 3 of Theorem~\ref{thm:7-cap-in-5-flat}, caps of size $7$ are maximal in dimension $5$, so $M(5) = 7$, and so $r_8$ must be at least $8$. Indeed, an $8$-cap in dimension $6$ can be constructed from any $7$-cap $C$ in dimension $5$ by adding any point outside of the $5$-flat containing $C$, since $\qc(C)$ is contained in this $5$-flat. Since $r_8 = 6$, we know $r_9$ must be at least $6$. In the next section we will show that $8$-caps need not be maximal, which proves that 9-caps can exist in dimension $6$, so $r_9 = 6$. Examples are given in Figures~\ref{fig:7-cap-dim-5} and~\ref{fig:89-caps-dim-6} of a $7$-cap in dimension $5$, an $8$-cap in dimension $6$, and a $9$-cap in dimension $6$, respectively.
\end{proof}
\end{corollary}

\begin{figure}[htbp]
    \centering
    \includegraphics[width=5cm]{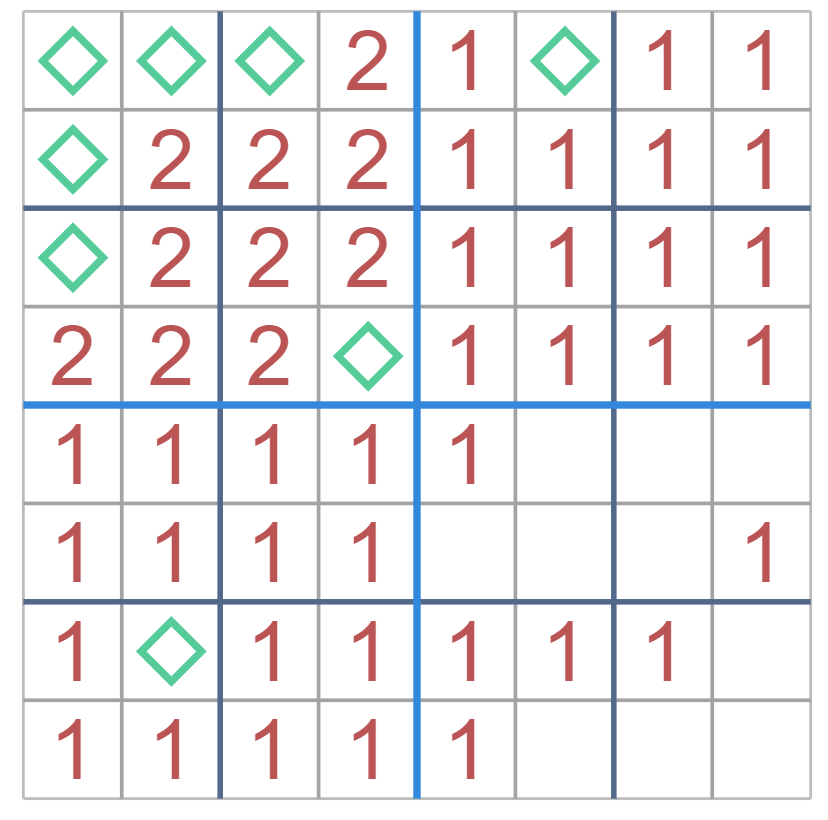}
    \qquad
    \includegraphics[width=5cm]{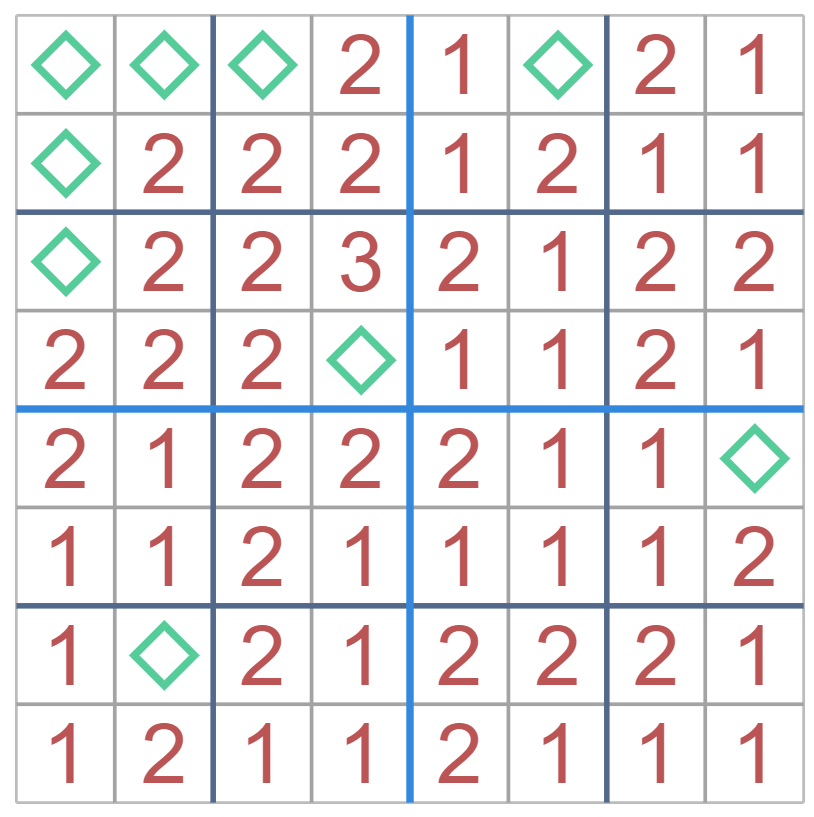}
    \caption{An $8$-cap (left) and a $9$-cap (right) in $\mathbb{Z}_2^6$.}
    \label{fig:89-caps-dim-6}
\end{figure}

\subsection{Maximal cap sizes and minimal cap dimensions}

We conclude this section with a summary of our results so far, in addition to the values of $r_{10}$ and $M(6)$, which will be derived in the next section. These are given in Tables~\ref{tab:k-cap-min-dims} and~\ref{tab:max-cap-sizes}.

\begin{table}[htbp]
    \centering
    \caption{The smallest dimension $r_k$ of a flat that admits a $k$-cap.}
    \label{tab:k-cap-min-dims}
    \begin{tabular}{|c||*{10}{c|}}
    \hline
    $k$ & 1 & 2 & 3 & 4 & 5 & 6 & 7 & 8 & 9 & 10 \\
    \hline
    $r_k$ & 0 & 1 & 2 & 3 & 4 & 4 & 5 & 6 & 6 & 7 \\
    \hline
    \end{tabular}
\end{table}

\begin{table}[htbp]
    \centering
    \caption{The maximal size $M(r)$ of a cap in dimension $r$.}
    \label{tab:max-cap-sizes}
    \begin{tabular}{|c||*{6}{c|}}
    \hline
    $r$ & 1 & 2 & 3 & 4 & 5 & 6 \\
    \hline
    $M(r)$ & 2 & 3 & 4 & 6 & 7 & 9 \\
    \hline
    \end{tabular}
\end{table}

\section{Equivalence Classes of Caps} \label{sec:equivalenceclasses}

In this section we complete our characterization of caps of sizes 1--9, and use this to partition $k$-caps into affine equivalence classes. 

We say that two caps in $\mathbb{Z}_2^n$ are \deff{equivalent} if they are affinely equivalent as subsets of $\mathbb{Z}_2^n$. Note that equivalent caps have the same cardinality, the same multiplicities of exclude points, and the same dimension. 

We begin by proving general results about equivalence classes for $k$-caps of dimensions $k-1$ and $k-2$. Recall that the dimension of a $k$-cap is at most $k-1$. 
We follow this with a characterization of equivalence classes of $k$-caps for all $k \leq 9$. Once $k \geq 10$, the analysis becomes much more complicated, and at this point we only have partial results. Fortunately, all $k$-caps in $AG(6,2)$ are of size $k \leq 9$, so these results are sufficient to completely characterize all caps in a Quad-64 deck.
 
In Table~\ref{tab:equiv-classes}, we give the number of equivalence classes of $k$-caps in $AG(n,2)$ along with their dimensions for $k \leq 9$ and $n \geq r_k$. Proofs of the results in the table are given in the remainder of this section.

\begin{table}[htbp]
    \centering
    \caption{Dimensions of all equivalence classes of $k$-caps for $1 \leq k \leq 9$.}
    \label{tab:equiv-classes}
    \begin{tabular}{|c|c|}
    \hline
    k & \text{Dimensions of Equivalence Classes of $k$-Caps} \\
    \hline \hline
    1 & \text{0} \\ \hline
    2 & \text{1} \\ \hline
    3 & \text{2} \\ \hline
    4 & \text{3} \\ \hline
    5 & \text{4} \\ \hline
    6 & \text{4, 5} \\ \hline
    7 & \text{5, 6} \\ \hline
    8 & \text{6 (two classes), 7} \\ \hline
    9 & \text{6, 7 (two classes), 8} \\
    \hline
    \end{tabular}
\end{table}

\subsection{Caps with at most one dependence relation}

We start by classifying equivalence classes of $k$-caps of dimension $k-1$, which are affinely independent, and those of dimension $k-2$, which have exactly one dependence relation. In the first case, there is only one equivalence class. In the second case, there is an equivalence class for each odd number between $5$ and $k-1$ --- note that this can only occur when $k \geq 6$.

Recall that in $AG(n,2)$, if $n < r_k$ then there cannot be a $k$-cap, so in what follows we assume that $n \geq r_k$.


\begin{theorem} \label{thm:k-1-equiv-classes} \leavevmode
\begin{enumerate}
    \item For any $k$, there is exactly one equivalence class of $k$-caps of dimension $k-1$.
    \item For $k \geq 6$, the $k$-caps of dimension $k-2$ can be partitioned into $\floor*{\frac{k-2}{2}} - 1$ non-empty equivalence classes, one for each odd integer $2m+1$ such $5 \leq 2m+1 \leq k-1$. Caps in the equivalence class corresponding to $2m+1$ contain a dependence relation of $2m+2$ points.
\end{enumerate}

\begin{proof}
By Proposition~\ref{prop:complete-maximal-cap-size}, every $k$-cap of dimension $k-1$ consists of $k$ affinely independent points, and therefore any two are equivalent. Hence for each $k$ there is only one equivalence class of $k$-caps of dimension $k-1$.

Let $C$ and $D$ be $k$-caps of dimension $k-2$ for some $k \geq 6$. Then there exist affinely independent caps $C' \subseteq C$ and $D' \subseteq D$ of size $k-1$, and $C - C'$ and $D - D'$ contain just one point each, say $x \in C - C'$ and $y \in D - D'$. By Lemma~\ref{lem:5-or-more}, $x$ is the sum of an odd number $2 \ell + 1$ of elements of $C'$ and $y$ is the sum of an odd number $2m+1$ elements of $D'$, where $2 \ell + 1$ and $2m + 1$ are between $5$ and $k-1$.
We will show that $C \cong D$ if and only if $\ell = m$. 

\begin{description}
\item[$(\Rightarrow)$] Suppose $C \cong D$ via an affine isomorphism $A : \mathbb{Z}_2^n \to \mathbb{Z}_2^n$ with $A(C) = D$. Since $A$ preserves affine independence and affine combinations, we know $A(C')$ is an affinely independent subset of $D$ of size $k-1$, and $A(x)$ is the sum of $2\ell + 1$ elements of $A(C')$. 

Write $x = x_1 + \cdots + x_{2\ell+1}$ for distinct $x_1, \ldots, x_{2\ell+1} \in C'$, and let $x_{2\ell+2}, \ldots, x_{k-1}$ be the remaining elements of $C'$. Then $A(x) = A(x_1) + \cdots + A(x_{2\ell+1})$, so $A(x) + A(x_1) + \cdots + A(x_{2\ell+1}) = \vec{0}$, which means for $1 \leq i \leq 2\ell+1$ any $A(x_i)$ is the sum of the other $2\ell+1$ elements. Then $y$ must be an element of the dependent set $A(C') = \{A(x), A(x_1), \ldots, A(x_{2\ell+1}), \ldots, A(x_{k-1}) \}$, for otherwise the affinely independent set $D' = D - \{y\}$ will contain the \emph{dependent} set $\{ A(x), A(x_1), \ldots, A(x_{2\ell+1}) \}$. This is a contradiction. Therefore $y = A(x)$ or $y = A(x_i)$ for some $1 \leq i \leq 2\ell+1$, so $y$ is the sum $2\ell+1$ elements of $D = A(C)$, as noted above. Note that if $\ell \neq m$, then $y$ is also the sum of $2m+1$ elements of $D$, and  equating the two sums will give an affine dependency in $D'$, which is impossible. So it must be that $2 \ell + 1 = 2m + 1$ and hence $\ell = m$.


\item[$(\Leftarrow)$] Suppose $\ell = m$. Let $x_1, \ldots, x_{2m+1} \in C'$ and $y_1, \ldots, y_{2m+1} \in D'$ be elements such that $x = x_1 + \cdots + x_{2m+1}$ and $y = y_1 + \cdots + y_{2m+1}$. Since $C'$ and $D'$ are affinely independent, we know there is some affine isomorphism from $\mathbb{Z}_2^n$ to $\mathbb{Z}_2^n$ such that $x_i \mapsto y_i$ for $1 \leq i \leq 2m+1$. Since affine tranformations preserve affine combinations, it follows that the isomorphism also maps $x = x_1 + \cdots + x_{2m+1}$ to $y = y_1 + \cdots + y_{2m+1}$. Therefore $C \cong D$ under this isomorphism.

\end{description}

Now we will show that each equivalence class is nonempty. Let $2m+1$ be an odd integer with $6 \leq m \leq k-1$. Let $C'$ be an affinely independent set of $k-1$ elements, so $C'$ is a $(k-1)$-cap of dimension $k-2$, and let $F = \aff(C')$. Let $y$ be the sum of $2m+1$ elements of $C'$. Then $y \in F$ so $C = C' \cup \{y\} \subseteq F$. We will prove that $C$ is a $k$-cap of dimension $k-2$. It suffices to show that $y \notin \exc(C')$.

Suppose that $y \in \exc(C')$. Then $y$ is the sum of both three elements of $C'$ and $2m+1 > 3$ elements of $C'$. By the uniqueness of distinct affine combinations of affinely independent elements, this is a contradiction. It follows that $y \notin \exc(C')$, so $C = C' \cup \{y\}$ is a $k$-cap with dimension $k-2$ contained in the specified equivalence class.

Finally, note that if $y \in C$ is the sum of $2m+1$ other points in $C$, then these $2m+2$ points sum to $\vec{0}$ so form a dependence relation.
\end{proof}
\end{theorem}

The proof above contained the following result.

\begin{corollary} \label{coro:unique-odd}
Let $C$ be a $(k-2)$-dimensional $k$-cap, where $k \geq 6$. Then there is a unique integer odd integer $5 \leq 2m+1 \leq k-1$ such that for any affinely independent subset $C' \subseteq C$ of size $k-1$, the singleton in $C - C'$ is a sum of $2m+1$ elements of $C'$. Equivalently, the only dependence relation in $C$ contains $2m+2$ points.
\end{corollary}

The situation for $k$-caps of dimension greater than or equal to $k-3$ is more complicated in general. We will see that for $k \leq 8$, all $k$-caps have dimension $k-1$ or $k-2$. It turns out that if we only want to analyze caps in Quad-64, it is sufficient to understand $k$-caps for $1 \leq k \leq 9$. As $k=9$ is the only case where there can be a cap of dimension $k-3$, we will only develop the theory of $k-3$ dimensional caps in this context.

\subsection{Equivalence classes of caps of sizes 1--8}

At this point, the hard work done in the previous sections will pay off. We have everything we need to classify caps of size up to $8$ up to affine equivalence. After this we will classify $9$-caps, the one remaining case that occurs in a Quad-64 deck.  

\begin{theorem}[Equivalence classes of $k$-caps for $k \leq 8$]
\label{thm:cap-classes}
\leavevmode
\begin{enumerate}
    \item For $1 \leq k \leq 5$, there is exactly one equivalence class of $k$-caps.
    \item There are two equivalence classes of $6$-caps, one for each cap dimension.
    \item There are two equivalence classes of $7$-caps, one for each cap dimension.
    \item There are three equivalence classes of $8$-caps; one for dimension $7$ and two for dimension $6$.
\end{enumerate}

\begin{proof}
By Theorem~ \ref{thm:k-1-equiv-classes}, any two $k$-caps of dimension $k-1$ are equivalent. Since the all caps of size  $k \leq 5$ are of dimension $k-1$ , there is only one equivalence class for these values of $k$. 

We've shown previously that $r_k = k-2$ for $6 \leq k \leq 8$, so caps of these sizes can occur in only two dimensions, $k-1$ and $k-2$. In the first case, $k-1$, we have one equivalence class for each $k$.

In the second case, according to \ref{thm:k-1-equiv-classes}, there is one non-empty equivalence class of $k$-caps of dimension $k-2$ for each odd integer $2m+1$ such that $5 \leq 2m+1 \leq k-1$. When $k=6$, we have $5 \leq 2m+1 \leq 6-1 = 5$, so there is only one equivalence class, and it corresponds to a dependence relation where one cap point is the sum of the remaining $5$.
When $k=7$, because $5 \leq 2m+1 \leq 7-1 = 6$, there is also only one equivalence class in dimension $5$, and it corresponds to a dependence relation where one cap point is the sum of $5$ other cap points.
When $k=8$, we have $5 \leq 2m+1 \leq 8-1 = 7$, so there are two equivalence classes in dimension $6$. In one case, there is a cap point that is the sum of $5$ others, and in the other there is a cap point that is the sum of the remaining $7$ points.
\end{proof}
\end{theorem}

In Figures~\ref{fig:caps-3-5}, \ref{fig:6-caps-types}, \ref{fig:7-caps-types},  and \ref{fig:8-caps-types}, we provide an example of a $k$-cap in each equivalence class, for $3 \leq k \leq 8$.

\begin{figure}[htbp]
    \centering
    \includegraphics[scale=0.14]{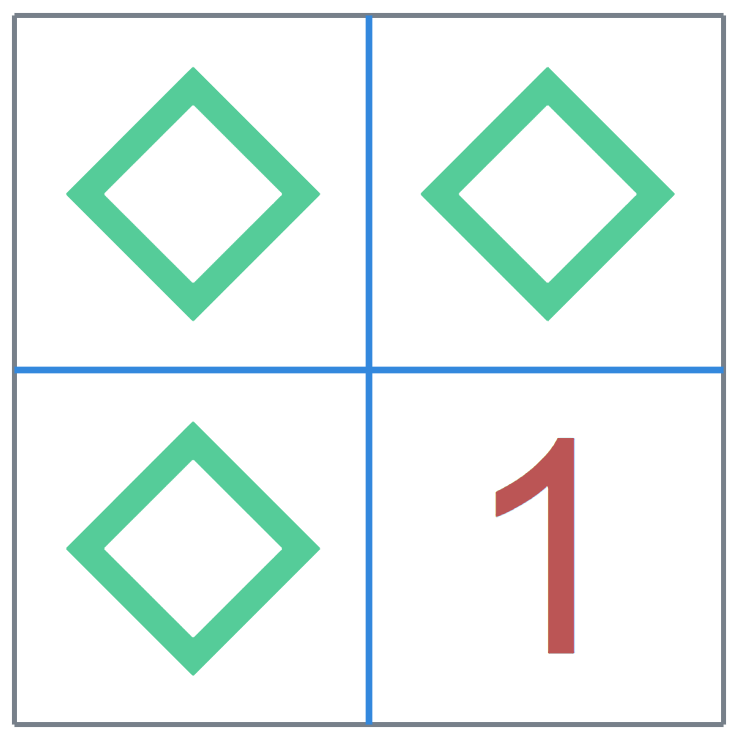}
    \qquad
    \includegraphics[scale=0.15]{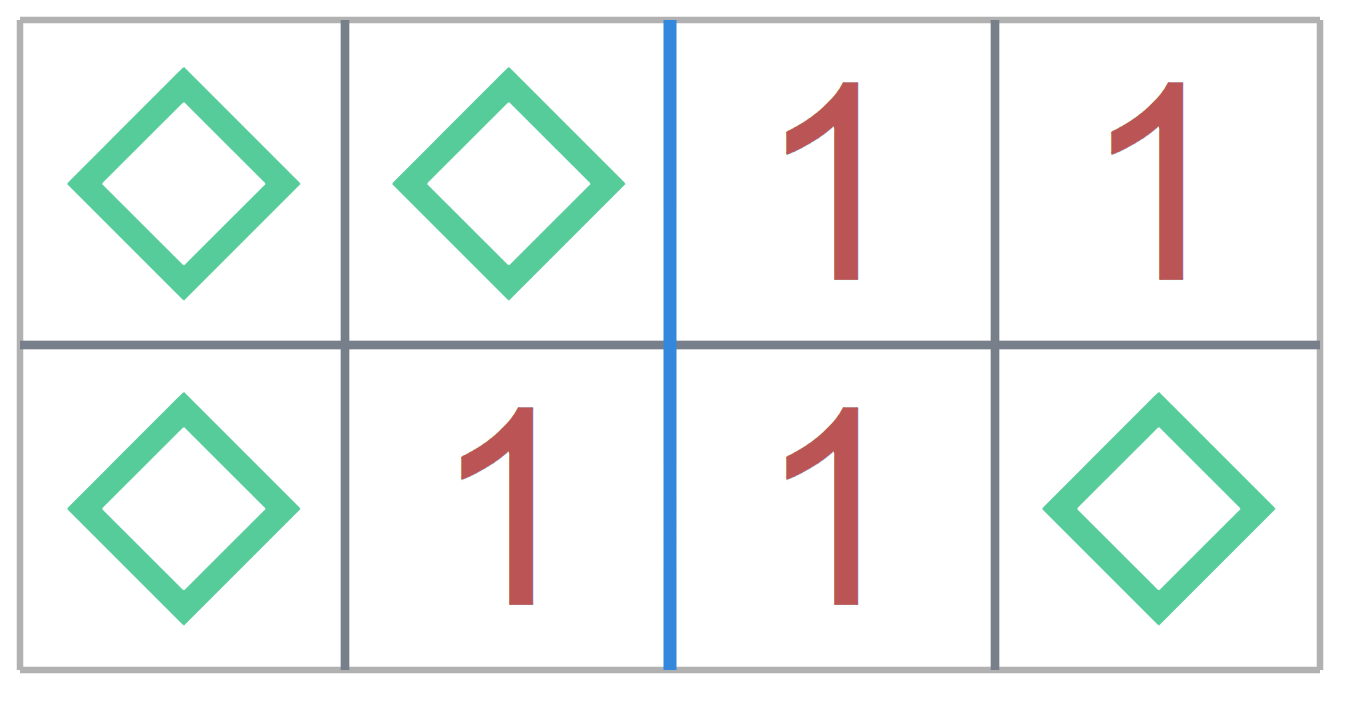}
    \qquad
    \includegraphics[scale=0.20]{images/5-cap-bottom-right.png}
    \caption{Caps of sizes 3, 4, and 5.}
    \label{fig:caps-3-5}
\end{figure}

\begin{figure}[htbp]
    \centering
    \includegraphics[scale=0.23]{images/6-cap.png}
    \qquad
    \includegraphics[scale=0.23]{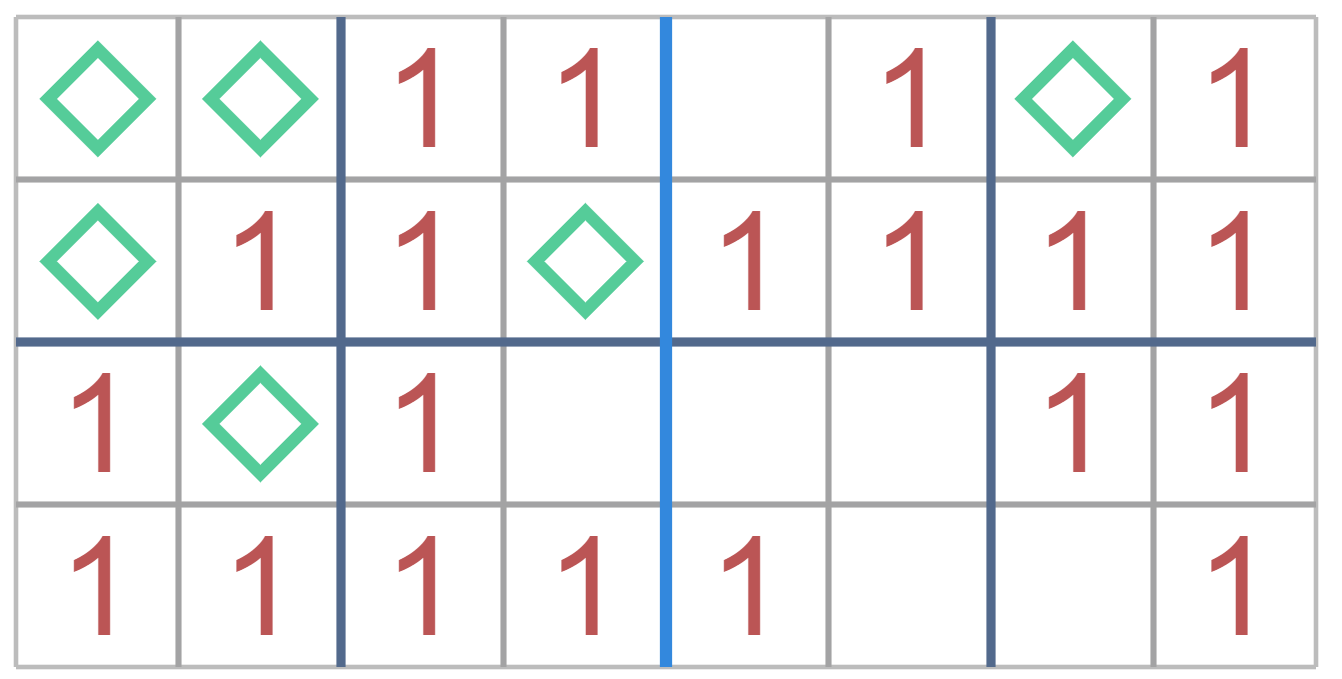}
    \caption{Cap in each of the two equivalence classes of $6$-caps.}
    \label{fig:6-caps-types}
\end{figure}

\begin{figure}[htbp]
    \centering
    \includegraphics[scale=0.17]{images/7-cap-dim-5.png}
    \qquad
    \includegraphics[scale=0.3]{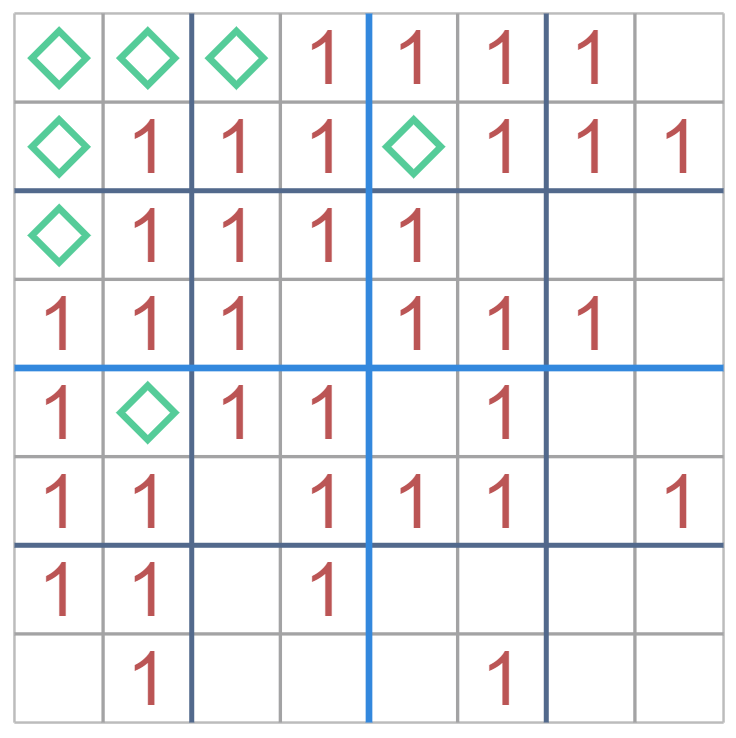}
    \caption{Caps in each of the two equivalence classes of $7$-caps.}
    \label{fig:7-caps-types}
\end{figure}

\begin{figure}[htbp]
    \centering
    \includegraphics[scale=0.34]{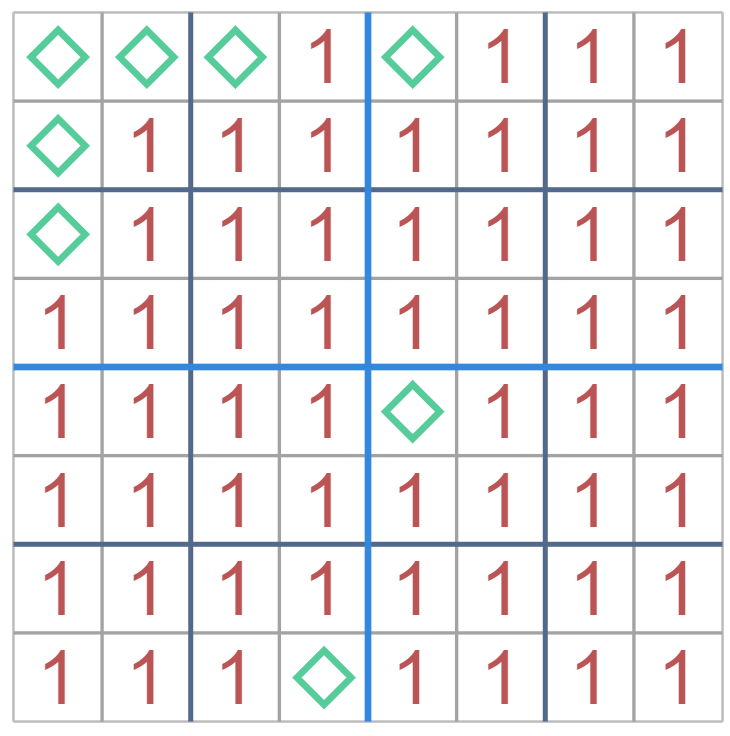}
    \qquad
    \includegraphics[scale=0.3]{images/8-cap-dim-6.png}
    
    \bigskip
    \includegraphics[scale=0.35]{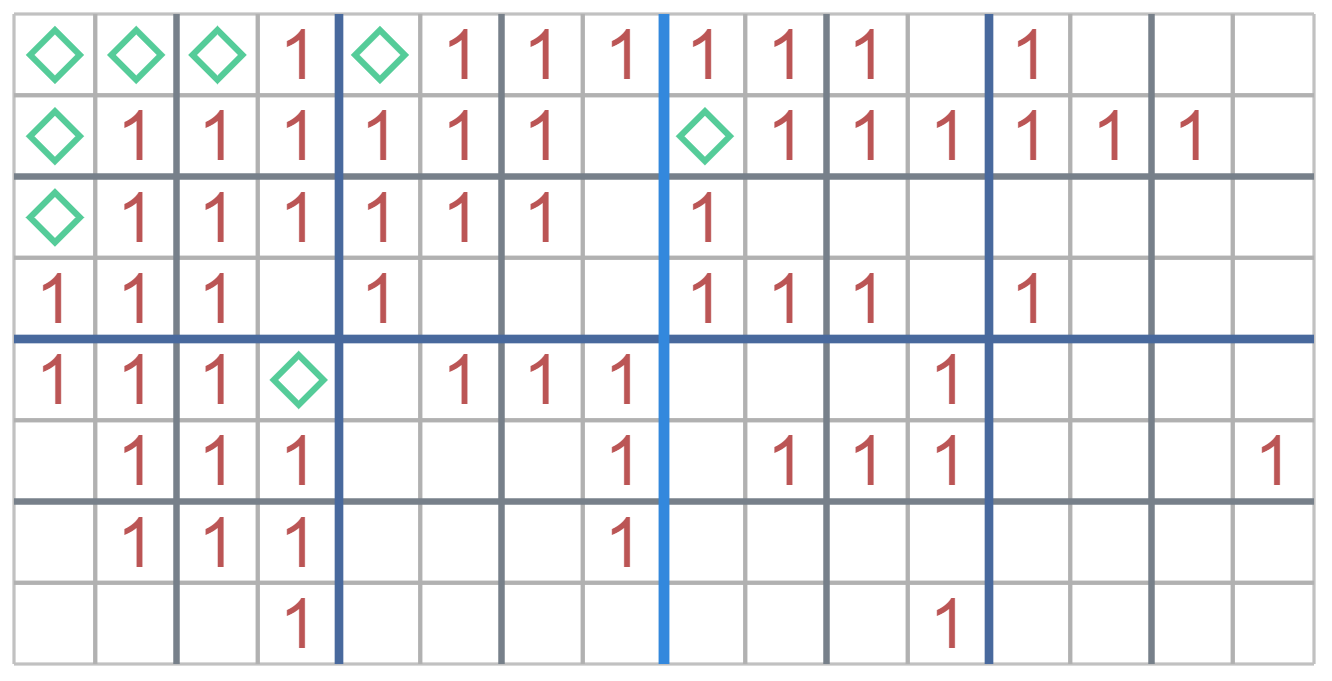}
    \caption{Caps in each of the 3 equivalence classes of $8$-caps.}
    \label{fig:8-caps-types}
\end{figure}

\subsection{Characterizing 8-caps}

In the proof of statement 3 of Theorem~{thm:cap-classes}, we showed that there are two equivalence classes of $8$-caps dimension $6$. In one case, the caps have a $6$-point dependence relation and in the other case, the caps have an $8$-point dependence relation. In both cases, these are the only dependencies in the caps.
In the following theorem, we say a bit more about the structure of caps in each class. This will be used later when we analyze classes of $9$-caps.

\begin{theorem}[Properties of 8-caps] \label{thm:8-classes}
Let $C$ be an $8$-cap of dimension $6$.
\begin{enumerate}
    \item If $C$ contains a $6$-point dependence, then there exists a $2$-point $p \in \exc(C)$.
    \item If $C$ contains an $8$-point dependence, then $C$ is complete and every $p \in \exc(C)$ is a $1$-point.
\end{enumerate}

\begin{proof}
Let $C = \{x_1, \ldots, x_8\}$.

Suppose $C$ contains a $6$-point dependence, say $x_1 + \cdots + x_6 = \vec{0}$. These $6$ points form a cap of dimension $4$, and by Corollary~\ref{cor:mult-2-or-higher} it follows that $C$ has an exclude point of multiplicity at least $2$. By Corollary~\ref{cor:s-geq-3k}, $C$ does not have enough elements to have an exclude point of multiplicity $3$ or higher, so $\exc(C)$ contains a $2$-point.

Suppose $C$ contains an $8$-point dependence, say $x_1 + \cdots + x_8 = \vec{0}$. Suppose there is a point $p \in \exc(C)$ with multiplicity greater than $1$. Relabeling if necessary, we can write $p = x_1 + x_2 + x_3 = x_4 + x_5 + x_6$. Then $x_1 + \cdots + x_6 = \vec{0}$. Adding this to the equation $x_1 + \cdots + x_8 = \vec{0}$ yields $x_{1} + x_{2} = \vec{0}$, so $x_{1} = x_{2}$, which is a contradiction. Thus every element of $\exc(C)$ is a $1$-point. But this means that

\[
    \lvert \qc(C) \rvert
    = \lvert C \rvert + \lvert \exc(C) \rvert
    = 8 + \binom{8}{3}
    = 8 + 56
    = 64.
\]
Since a $6$-flat has exactly $2^6 = 64$ points, if follows that $\qc(C) = \aff(C)$, so $C$ is complete in dimension $6$.  
\end{proof}
\end{theorem}

\begin{remark}
In the next subsection we show that there are $6$-dimensional $9$-caps, so the $6$-dimensional $8$-caps cannot be maximal.
\end{remark}

\subsection{Characterization of 9-caps}

We now move on to $9$-caps. There are $4$ equivalence classes of $9$-caps, starting in dimension $r_9 = 6$.

\begin{theorem}[Equivalence classes of 9-caps] \label{thm:9-classes}
There are $4$ equivalence classes of $9$-caps $C$:
\begin{enumerate}
    \item $9$-caps of dimension $8$, which are affinely independent.
    \item $9$-caps of dimension $7$ with an $8$-point dependence and only $1$-points in $\exc(C)$. 
    \item $9$-caps of dimension $7$ with a $6$-point dependence and at least one $2$-point in $\exc(C)$.
    \item $9$-caps of dimension $6$, which are maximal and have a $3$-point in $\exc(C)$. 
\end{enumerate}
\end{theorem}

\begin{proof}
Let $C$ be a $9$-cap. By Proposition~\ref{prop:complete-maximal-cap-size}, $r_9 \leq \dim(C) \leq 8$. Since $r_9 = 6$ by Corollary~\ref{cor:max-cap-dim-5}, the dimension of a $9$-cap can only be $6$, $7$, or $8$.


\textbf{Dimension 8:} By Theorem~ \ref{thm:k-1-equiv-classes} all $9$-caps of dimension $8$ are affinely independent and equivalent.

\textbf{Dimension 7:} By Theorem~\ref{thm:k-1-equiv-classes} there is an equivalence class of $7$-dimensional $9$-caps for each odd integer $2m+1$ with $5 \leq 2m+1 \leq 9-1 = 8$.
The applicable odd integers are $5$ and $7$, so there are two non-empty equivalence classes of $9$-caps of dimension $7$. In one class the caps have a dependence relation of $6$ points, and in the other the caps have a dependence relation of $8$ points.
Let $C = \{x_1, \ldots, x_9\}$ be a $9$-cap of dimension $7$.

\begin{description}
    \item[Case 1] Suppose $C$ contains a $8$-point dependence. Relabeling if necessary, we can write it as $x_1 + \cdots + x_8 = \vec{0}$. Then these points form a $6$-dimensional $8$-cap $D$ with an $8$-point dependence relation. By Theorem~\ref{thm:8-classes}, $D$ completes a $6$-flat $F$, and every point in $\exc(D)$ has multiplicity $1$. It follows that $x_9 \in C - D$ must be outside $F$.
    
    Now, suppose there is $p \in \exc(C)$ with multiplicity at least $2$. Then $p$ is the sum of two different triples, one of which contains $x_9$, since elements of $\exc(D)$ are $1$-points. Relabeling if necessary, we can write $p = x_1 + x_2 + x_3 = x_4 + x_5 + x_9$.
    Then $x_{1} + \cdots + x_{5} + x_9 = \vec{0}$. Adding this to the equation $x_1 + \cdots + x_8 = \vec{0}$ yields $x_6 + x_7 + x_8 + x_9 = \vec{0}$, which contradicts the fact that $C$ is a cap. Thus every element of $\exc(C)$ is a $1$-point.

    \item[Case 2] Suppose $C$ contains a $6$-point dependence. Relabeling if necessary, we can write it as $x_1 + \cdots + x_6 = \vec{0}$. Then these points form a $4$-dimensional $6$-cap. By Corollary~\ref{cor:mult-2-or-higher} it follows that there is a point in $\exc(C)$ with multiplicity $2$ or higher. 
    
    Suppose there exists $p \in \exc(C)$ with multiplicity at least $3$.  
    Then
    \[
        p
        = a_1 + b_1 + c_1
        = a_2 + b_2 + c_2
        = a_3 + b_3 + c_3,
    \]
    for nine distinct elements, and hence all elements, of $C$.
    It follows that $c_2 = a_1 + b_1 + c_1 + a_2 + b_2$ and $c_3 = a_1 + b_1 + c_1 + a_3 + b_3$. Since $c_2$ and $c_3$ are each sums of an odd number of elements, we have $c_2, c_3 \in \aff \{a_1, b_1, c_1, a_2, b_2, a_3, b_3 \}$. Therefore $C$ is spanned by $7$ elements, so $\dim(C) \leq 6$, contradicting the assumption that $\dim(C) = 7$. Therefore $\exc(C)$ cannot contain a $3$-point, so it must contain a $2$-point.
\end{description}

\textbf{Dimension 6:} Let $C = \{x_1, \ldots, x_9\}$ be a $9$-cap of dimension $6$, and let $F = \aff(C)$ be the $6$-flat spanned by $C$. Then $C$ contains an affinely independent subset, say, $C' = \{x_1, \ldots, x_7\}$, and $x_8$ and $x_9$ are each affine combinations of elements of $C'$. It follows that each of $x_8$ and $x_9$ is a sum of either $5$ or $7$ of points in $C'$. We will show that neither can be the sum of $7$ points in $C'$.

Suppose one of them is a sum of $7$ points --- for instance, that $x_8 = x_1 + \cdots + x_7$. Then $D =\{x_1, \ldots, x_8\}$ is an $8$-cap of dimension $6$ with an $8$-point dependence relation. By Theorem~\ref{thm:8-classes} we know $D$ is a complete cap in dimension $6$, so the cap cannot be expanded in dimension $6$ to include $x_9$. This is a contradiction. Therefore both $x_8$ and $x_9$ must be sums of $5$ elements of $C'$.

Relabeling if necessary, we can write $x_8 = x_1 + \cdots + x_5$. We claim that $x_9$ is the sum of $x_6$ and $x_7$ and three additional elements of $C'$.

Note that the $5$-point sum for $x_9$ must involve at least one of $x_6$ and $x_7$, since $x_9 \neq x_8 = x_1 + \cdots + x_5$. Suppose the $5$-point sum for $x_9$ involves $x_6$ but not $x_7$. Then without loss of generality, $x_9 = x_1 + \cdots + x_4 + x_6$. Adding this to $x_8 = x_1 + \cdots + x_5$ and simplifying, we obtain $x_9 + x_8+ x_5 + x_6 = \vec{0}$, which contradicts the fact that $C$ is a cap. Hence the sum must involve both $x_6$ and $x_7$. 

Relabeling if necessary, we can write $x_9 = x_1 + x_2 + x_3 + x_6 + x_7$.

We now show that $C$ completes the $6$-flat $F$. Since $x_8 = x_1 + \cdots + x_5$ we know that $\{x_1, \ldots, x_5, x_8\}$ is a $6$-cap contained in a $4$-flat $F_1 \subseteq F$. By Theorem~\ref{thm:cap-classes} this $6$-cap is complete in $F_1$, so the other points in $C$ must lie outside of $F_1$. Let $\{F_1, F_2, F_3, F_4\}$ be a partition of the $F$ into pairwise disjoint $4$-flats. By Proposition~\ref{prop:1-per-4-flat}, we know that each of $F_2, F_3, F_4$ can contain at most one point of $C$.  Then $x_6, x_7$, and $x_9$ must each lie in a different $4$-flat, one of $F_2, F_3, F_4$. Since no other points can be added to $C$ to form a larger cap, $C$ must be complete.  

To show that $\exc(C)$ has a point of multiplicity $3$, observe that the equations
\begin{align*}
    x_8 &= x_1 + x_2 + x_3 + x_4 + x_5, \\[1ex]
    x_9 &= x_1 + x_2 + x_3 + x_6 + x_7
\end{align*}
imply that
\[
    x_1 + x_2 + x_3 = x_4 + x_5 + x_8 = x_6 + x_7 + x_9.
\]
so $p = x_1 + x_2 + x_3$ has multiplicity at least $3$. But by Corollary~\ref{cor:s-geq-3k}, $C$ does not have enough elements for there to be a point in $\exc(C)$ with multiplicity $4$ or higher, so we conclude that $p$ is a $3$-point.

To show that all $6$-dimensional $9$-caps are equivalent, let $D$ be another $9$-cap of dimension $6$. By our work above, we know we can write $D = \{y_1, \ldots, y_9\}$ where:
\begin{itemize}
    \item $D' = \{y_1, \ldots, y_7\}$ is affinely independent;
    \item $y_8 = y_1 + \cdots + y_5$;
    \item $y_9 = y_1 + y_2 + y_3 + y_6 + y_7$.
\end{itemize}
Since $C' = \{x_1, \ldots, x_7\}$ and $D'$ are affinely independent subsets of the same size, there is an affine isomorphism $A : \mathbb{Z}_2^n \to \mathbb{Z}_2^n$ such that $A(x_i) = y_i$ for $1 \leq i \leq 7$. Since $A$ preserves affine combinations, it follows that $A(x_8) = y_8$ and $A(x_9) = y_9$ as well. Therefore $C \cong D$ under the affine isomorphism $A$. 

Now that we know that all $9$-caps of dimension $6$ are equivalent, they must all be complete caps, so by Proposition~\ref{prop:complete-maximal-cap-size} they are all maximal.

\end{proof}

\begin{example}
In Figure~\ref{fig:9-caps-types} we give examples of each of the four equivalence classes of $9$-caps. There is one of dimension $6$, two of dimension $7$, and one of dimension $8$.

\begin{figure}[htbp]
    \centering
    \includegraphics[scale=0.33]{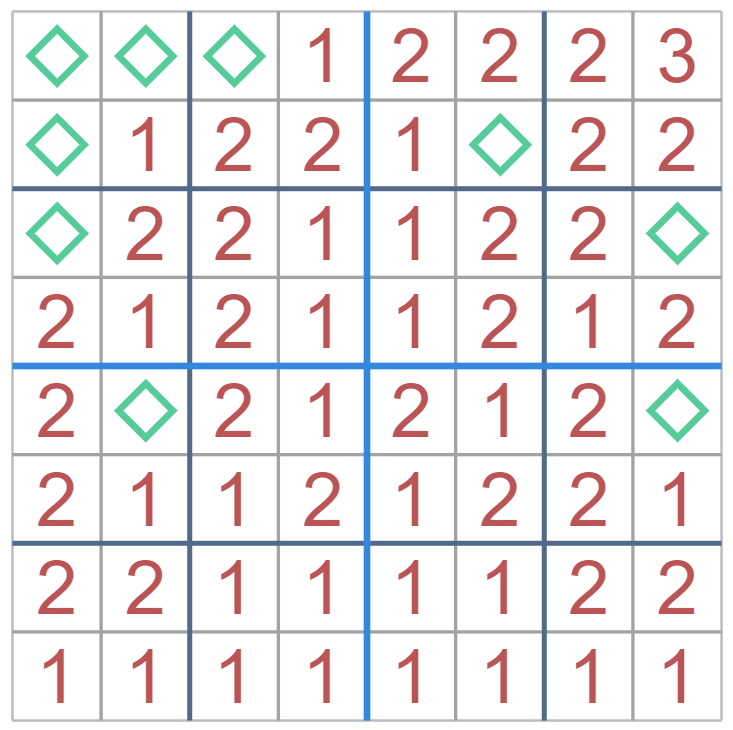}
    
    \bigskip
    
    \includegraphics[scale=0.33]{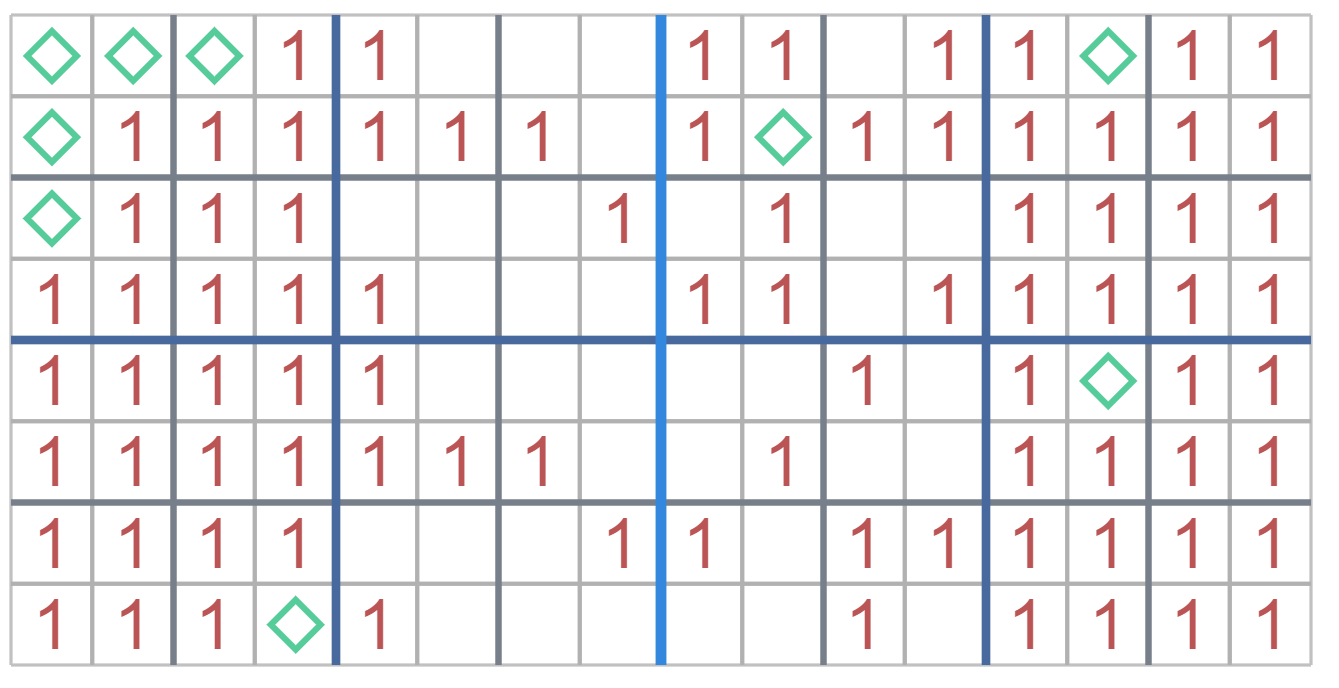}
    
    \bigskip

    \includegraphics[scale=0.33]{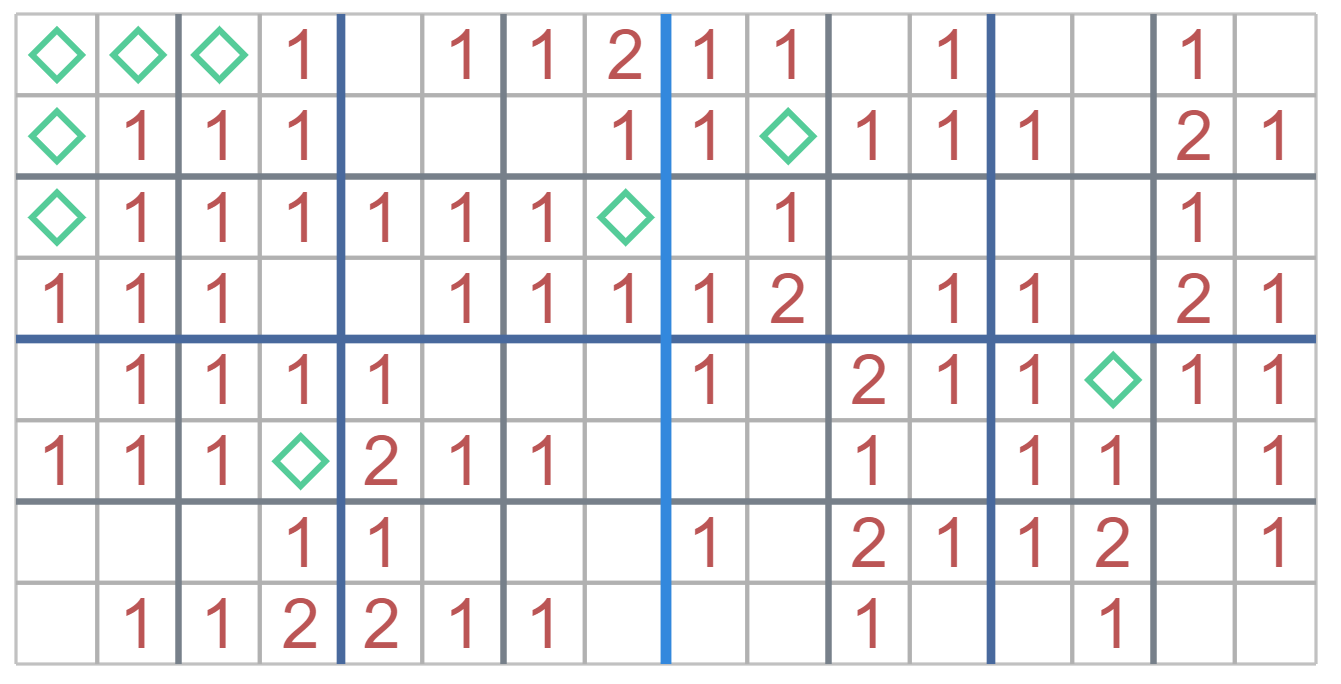}
    
    \bigskip
    
    \includegraphics[scale=0.43]{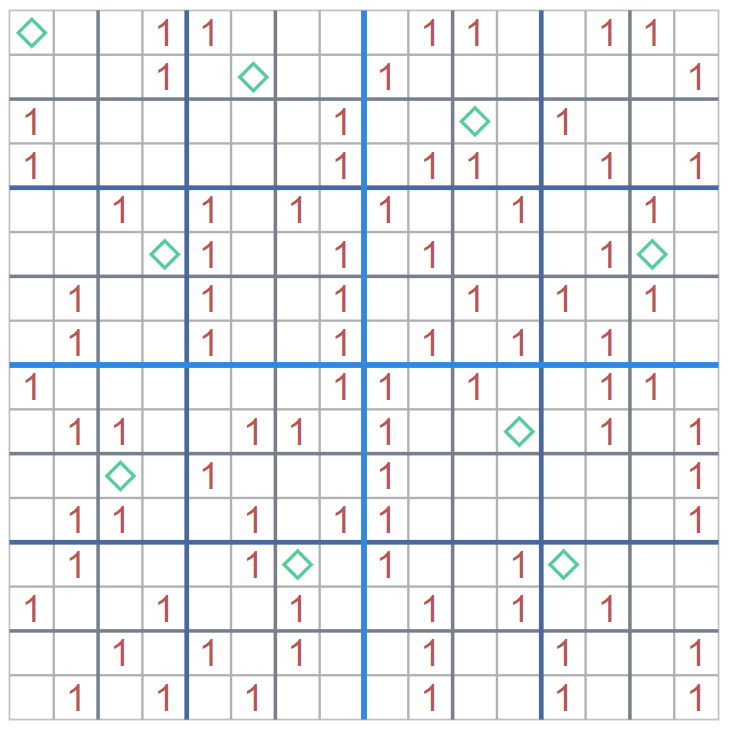}
    \caption{Examples of each of the 4 equivalence classes of $9$-caps.}
    \label{fig:9-caps-types}
\end{figure}
\end{example}

\begin{corollary} \label{cor:max-cap-dim-6}
The size of a maximal cap in dimension $6$ is $M(6) = 9$, and the smallest dimension containing a $10$-cap is $r_{10} = 7$.

\begin{proof}
This follows immediately from Proposition~\ref{prop:complete-maximal-cap-size}, Theorem~\ref{thm:9-classes}, and the example in Figure~\ref{fig:10-cap-dim-7} of a $10$-cap of dimension $7$.
\end{proof}
\end{corollary}

\begin{figure}[htbp]
    \centering
    \includegraphics[width=6cm]{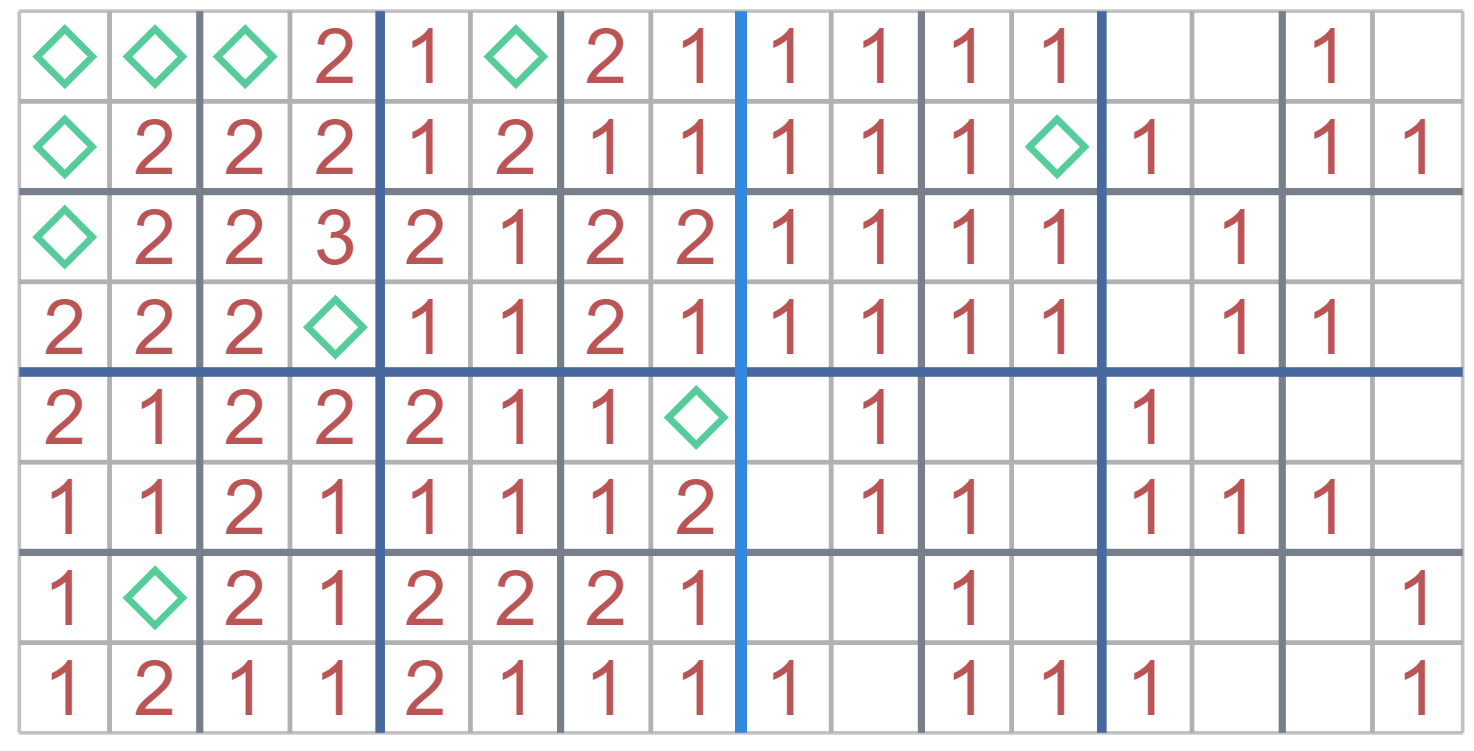}
    \caption{A $10$-cap in $\mathbb{Z}_2^7$.}
    \label{fig:10-cap-dim-7}
\end{figure}

Corollary~\ref{cor:max-cap-dim-6} implies that there cannot be a $k$-cap with $k \geq 10$ in $AG(6,2)$. Therefore any 10 cards in Quad-64 must contain a quad!

\section{Counting Caps} \label{sec:countingcaps}

In this section we use the results of the previous section to derive explicit formulas for the number of $k$-caps of each dimension for $k \leq 9$. As noted previously, this includes all possible caps in Quad-64.  We can then use this to compute the probability of a quad in a random initial $k$-card layout in Quad-16, Quad-32 and Quad-64. We will then be equipped to answer the burning question: 
 
\begin{quote}
\large{\textit{How many cards should you lay out
in a game of {EvenQuads}?}}
\end{quote}

\begin{definition}
Let $Q(n,k)$ denote the number of $k$-caps in $AG(n,2)$, and let $Q_{r}(n,k)$ denote the number of such $k$-caps of dimension $r$.
\end{definition}

Recall that $M(r)$ is the size of a maximal cap in dimension $r$, and observe that if $M(r) < k$ then $Q_{r}(n,k) = 0$. Recall also that $r_k$ is the smallest dimension of a flat that can contain a $k$-cap, so $r_k$ is the smallest value of $r$ for which $M(r) \geq k$. Note also that $r_k$ is at most $k-1$.

\begin{lemma}
If $n \geq r_k$, then
\[
    Q(n, k) = \sum_{r=r_k}^{k-1} Q_{r}(n, k).
\]
Otherwise $Q(n, k) = 0$.

\begin{proof}
The maximal dimension of the affine span of a $k$-cap is $k-1$, and $r_k$ is the minimal dimension for a $k$-cap to exist.
\end{proof}
\end{lemma}

Combined with our work in the previous section classifying equivalence classes of caps, we have the following proposition, which follows immediately from Theorems~\ref{thm:cap-classes} and~\ref{thm:9-classes}.

\begin{proposition} \label{prop:counting-cases}
For $1 \leq 9 \leq k$, we have
\[
    Q(k,n) =
    \begin{cases}
        Q_{k-1}(n,k) & \text{if $1 \leq k \leq 5$}, \\
        Q_{k-1}(n,k) + Q_{k-2}(n,k) & \text{if $6 \leq k \leq 8$}, \\
        Q_{k-1}(n,k) + Q_{k-2}(n,k) + Q_{k-3}(n,k) & \text{if $k=9$}.
    \end{cases}
\]

\end{proposition}

\subsection{Cap Counting Formulas}
We will now derive formulas for the number of $k$-caps of dimensions $k-1$ and $k-2$ for any $k$, and for the number of $9$-caps of dimension $6$. By Proposition~\ref{prop:counting-cases}, this is all that is required to count the number of $k$-caps with $1 \leq k \leq 9$.

Recall that all $k$-caps of dimension $k-1$ are affinely independent, and by Proposition~\ref{thm:ind-implies-cap} all affinely independent sets of size $k$ are $k$-caps of dimension $k-1$). We now give a general formula for the number of affinely independent $k$-caps. We assume throughout that $n \geq r_k$.

\begin{theorem} \label{thm:k-1-dim-count}
 When $k=1$, $Q_{0}(1,n) = 2^n$, and for $k \geq 2$,
\begin{align*}
    Q_{k-1}(n,k)
    &= \frac{2^n}{k!} \cdot \prod_{i=0}^{k-2} (2^n - 2^i) \\[1ex]
    &= \frac{2^n (2^n-1) (2^n-2) (2^n-4) \cdots (2^n-2^{k-2})}{k!}
\end{align*}

\begin{proof}
First note that all singleton sets are affinely independent, so $Q_0(n,1) = 2^n$. We prove that the above formula holds for all $k \geq 2$ by induction on $k$.

Since all subsets of size $k=2$ are affinely independent,
\[
    Q_1(n,2)
    = \binom{2^n}{2}
    = \frac{2^n(2^n-1)}{2!},
\]
which matches the desired formula.

Now suppose $k \geq 2$ is an integer and the above formula is true for all affinely independent $k$-caps. Any affinely independent $(k+1)$-cap can be built from an affinely independent $k$-cap $C$, so we just need to count the number of ways we can add a point $x$ to $C$ so that $C \cup \{x\}$ is affinely independent. Recall that $\aff(C)$ contains $2^{k-1}$ points. The only restriction is that $x \notin \aff(C)$, so $x$ can be any of the $2^n - 2^{k-1}$ remaining points in $\mathbb{Z}_2^n$. For any $(k+1)$-cap there are $k+1$ distinct ways to view it as the union of a $k$-cap and an additional point, so this method of constructing $(k+1)$-caps from $k$-caps will produce each distinct $(k+1)$-cap exactly $k+1$ times. It follows that
\begin{align*}
    Q_k(n,k+1)
    &= Q_{k-1}(n,k) \cdot \frac{2^n-2^{k-1}}{k+1} \\[1ex]
    &= \left[ \frac{2^n}{k!} \cdot \prod_{i=0}^{k-2} (2^n - 2^{i}) \right] \cdot \frac{2^n-2^{k-1}}{k+1} \\[1ex]
    &= \frac{2^n}{(k+1)!} \cdot \prod_{i=0}^{k-1} (2^n - 2^{i}).
\end{align*}
Hence by induction, the formula is true for all $k \geq 2$.
\end{proof}
\end{theorem}

Now we count the number of $k$-caps of dimension $k-2$.

\begin{theorem} \label{thm:k-2-dim-count}
For $k \geq 6$, the number of $k$-caps of dimension $k-2$ is given by
\begin{align*}
    Q_{k-2}(n,k)
    &= Q_{k-2}(n,k-1) \cdot \sum_{i=2}^h \frac{1}{2i+2} \cdot \binom{k-1}{2i+1} \\[1ex]
    &= Q_{k-2}(n,k-1) \cdot
        \left[ \frac{1}{6} \binom{k-1}{5} + \frac{1}{8} \binom{k-1}{7} + \cdots + \frac{1}{2h+2} \binom{k-1}{2h+1} \right]
\end{align*}
where $h = \floor*{\frac{k-2}{2}}$.

\begin{proof}
Let $h = \floor*{\frac{k-2}{2}}$.
By Corollary~\ref{coro:unique-odd} and Theorem~\ref{thm:k-1-equiv-classes}, the set of $k$-caps of dimension $k-2$ can be partitioned into non-empty equivalence classes $\fanc_m$, one for each odd integer $2m + 1$ with $2 \leq m \leq h$. 
Hence
\[
    Q_{k-2}(n,k) = \sum_{m=2}^h \lvert \fanc_{m} \rvert,
\]

Let $m$ be an integer with $2 \leq m \leq h$.
Let $C' = \{x_1, \ldots, x_{k-1}\}$ be a $(k-1)$-cap of dimension $k-2$, which means it is affinely independent, let $S \subseteq C'$ be any $(2m+1)$-element subset, and let $x_k$ be the sum of the elements of $S$.
Note that the $x_k$ cannot be in $\exc(C')$ because $2m+1 > 3$.
Therefore the union $C = C' \cup \{x_k\}$ is a $k$-cap of dimension $k-2$.
Conversely, the proof of Theorem~\ref{thm:k-1-equiv-classes} demonstrates that every $k$-cap of dimension $k-2$ can be constructed in this way.
Since $Q_{k-2}(n,k-1)$ is the number of $(k-2)$-dimensional $(k-1)$-caps and $\binom{k-1}{2m+1}$ is the number of $(2m+1)$-element subsets, there are $Q_{k-2}(n,k-1) \cdot \binom{k-1}{2m+1}$ ways to perform this construction, and the results cover every possible $(k-2)$-dimensional $k$-cap at least once.
We will now show that each such cap occurs exactly $2m+2$ times by this construction.

Suppose $C = C' \cup \{x_k\}$ is as above. Then
\begin{equation} \label{eq:aff-dep}
    x_k + \sum_{s \in S} s = \vec{0}
\end{equation}
is the only affine dependence in $C$. It follows that the only way to obtain the same cap $C$ from our construction is to start with a different independent set of the form $(C' - \{s_0\}) \cup \{x_k\}$ for some $s_0 \in S$, and choose $(S - \{s_0\}) \cup \{x_k\}$ as our $(2m+1)$-element subset, noting that by equation~\eqref{eq:aff-dep} we have
\[
    s_0 = x_k + \sum_{s \in S - \{s_0\}} s.
\]
There are $2m+1$ elements $s \in S$ from which to choose, and together with our original construction of $C$, this gives exactly $2m+2$ ways to construct $C$.

Therefore
\[
    \lvert \fanc_m \rvert = \frac{Q_{k-2}(n,k-1) \cdot \binom{k-1}{2 m + 1}}{2m+2}
\]
and hence
\begin{align*}
    Q_{k-2}(n,k)
    = \sum_{m = 2}^h \lvert \fanc_m \rvert
    &= \sum_{m = 2}^h \frac{Q_{k-2}(n,k-1) \cdot \binom{k-1}{2m+1}}{2m+2} \\[1ex]
    &= Q_{k-2}(n,k-1) \cdot \sum_{m = 2}^h \frac{1}{2m+2} \cdot \binom{k-1}{2m+1}.
\end{align*}
\end{proof}
\end{theorem}

\begin{theorem} \label{thm:count-9-caps-dim-6}
The number of $9$-caps of dimension $6$ in $\mathbb{Z}_2^n$ is $Q_6(n,9) = \frac{35}{9} \cdot Q_6(n,7)$.

\begin{proof}
Let $C$ be a $9$-cap of dimension $6$. By the proof of part (d) of Theorem~\ref{thm:9-classes} we can write $C = \{x_1,\ldots,x_9\}$ where $C' = \{x_1, \ldots, x_7\}$ is affinely independent, $x_8$ and $x_9$ are each the sum of 5 elements of $C'$, and they have 3 elements of $C'$ in common. For example
\begin{align*}
    x_8 &= x_1 + x_2 + x_3 + x_4 + x_5, \\
    x_9 &= x_1 + x_2 + x_3 + x_6 + x_7.
\end{align*}
is one such solution.

We can construct such a cap as follows. Start with an independent set $C'$ of size 7. There are $Q_{6}(n,7)$ ways to do this. Now choose the 3 elements in common when we write $x_8$ and $x_9$ as sums of 5 elements of $C'$. There are $\binom{7}{3} = 35$ ways to do this. Finally, $\frac{1}{2} \cdot \binom{4}{2} = 3$ is the number of ways to split the remaining four elements of $C'$ into two pairs, one each to complete the sums for $x_8$ and $x_9$, respectively. 

Hence there are $35 \cdot 3 \cdot Q_{6}(n,7) = 105 \cdot Q_{6}(n,7)$ ways construct a 9-cap, but we have overcounted. We will show that each cap occurs exactly $27$ times by this construction. It suffices to count how many different independent sets $C'$ we could have started with to produce the same cap, with the same dependence relations. 

Assume our cap has the form $C = C' \cup \{x_8, x_9\}$, where $C' = \{x_1, \ldots, x_7\}$ and
\begin{align*}
    x_8 &= x_1 + x_2 + x_3 + x_4 + x_5, \\
    x_9 &= x_1 + x_2 + x_3 + x_6 + x_7.
\end{align*}

Then we have a triple point described by
\begin{equation} \label{eq:3-point}
    x_1 + x_2 + x_3
    = x_4 + x_5 + x_8
    = x_6 + x_7 + x_9.
\end{equation}

This divides the elements of $C$ into three $3$-element subsets; call them $Y_1, Y_2, Y_3$. In order to construct an affinely independent subset of size 7, we can take at most 5 points from any pair of sets $Y_i, Y_j$ (otherwise we will have a dependence relation), and the remaining 2 will come from the third set $Y_k$. There are 3 ways to choose one of $Y_1,Y_2,Y_3$, and $\binom{3}{2} \cdot \binom{3}{2} = 9$ ways to choose two elements each from the other sets, so there are $3 \cdot 9 = 27$ distinct affinely independent $7$-element subsets in $C$.

Thus the number of $6$-dimensional $9$-caps is
\[
    Q_6(n,9)
    = \frac{105}{27} \cdot Q_6(n,7)
    = \frac{35}{9} \cdot Q_6(n,7).
\]
\end{proof}
\end{theorem}

\subsection{An application to card layouts in \textit{EvenQuads}}

Using the formulas developed in this section (Proposition~\ref{prop:counting-cases} and Theorems~\ref{thm:k-1-dim-count}, \ref{thm:k-2-dim-count}, and \ref{thm:count-9-caps-dim-6}), we can now count the number of caps of various sizes and dimensions in a deck of \textit{EvenQuads} cards. These are given in Table~\ref{tab:cap-count}.

The number of subsets of size $k$ in $\mathbb{Z}_2^6$ is $\binom{2^6}{k} = \binom{64}{k}$, so the probability that a random subset of size $k$ in $\mathbb{Z}_2^6$ is a cap is
\[
    P(\text{$k$ cards is a cap}) = \frac{Q(k,6)}{\binom{64}{k}}
\]
and the probability that it contains at least one quad is
\[
    P(\text{$k$ cards contain a quad})
    = 1 - \frac{Q(6,k)}{\binom{64}{k}}.
\]
Table~\ref{tab:quad-probs} displays these computations for $k=1, \ldots, 10$.

\begin{table}[htbp]
\centering
\caption{Values of $Q_r(6,k)$, the number of $r$-dimensional caps of size $k$ in $AG(6,2)$.}
\label{tab:cap-count}
\[
\begin{array}{|c|c|c|c|c|}
    \hline
    k & Q_{k-1}(6,k) & Q_{k-2}(6,k) & Q_{k-3}(6,k) & Q(6,k) \\
    \hline \hline
    1 &              64 &               0 &               0 & 64
    \\ \hline
    2 &         2,\!016 &               0 &               0 & 2,\!016
    \\ \hline
    3 &        41,\!664 &               0 &               0 & 41,\!664
    \\ \hline
    4 &       624,\!960 &               0 &               0 & 624,\!960
    \\ \hline
    5 &   6,\!999,\!552 &               0 &               0 & 6,\!999,\!552
    \\ \hline
    6 &  55,\!996,\!416 &   1,\!166,\!592 &               0 & 57,\!163,\!008
    \\ \hline
    7 & 255,\!983,\!616 &  55,\!996,\!416 &               0 & 311,\!980,\!032
    \\ \hline
    8 &               0 & 927,\!940,\!608 &               0 & 927,\!940,\!608
    \\ \hline
    9 &               0 &               0 & 995,\!491,\!840 & 995,\!491,\!840
    \\ \hline
    10 & 0 & 0 & 0 & 0 \\ \hline
\end{array}
\]
\end{table}

\begin{table}[htbp]
\centering
\caption{The probability of a quad in a random layout of \textit{EvenQuads} cards.}
\label{tab:quad-probs}
\[
\begin{array}{|c|c|c|}
    \hline
    k & \text{Probability of no quads in $k$ cards} & \text{Probability of a quad in $k$ cards} \\
    \hline \hline
    1 & 1 & 0 \\
    \hline
    2 & 1 & 0 \\
    \hline
    3 & 1 & 0 \\
    \hline
    4 & 0.9836065574 & 0.01639344262 \\
    \hline
    5 & 0.9180327869 & 0.08196721311 \\
    \hline
    6 & 0.7624340094 & 0.2375659906 \\
    \hline
    7 & 0.5022084679 & 0.4977915321 \\
    \hline
    8 & 0.2096488791 & 0.7903511209 \\
    \hline
    9 & 0.03614635846 & 0.9638536415 \\
    \hline
    10 & 0 & 1 \\
    \hline
\end{array}
\]
\end{table}

The official instructions for \textit{EvenQuads} given at \href{https://awm-math.org/publications/playing-cards/EvenQuads}/{https://awm-math.org/publications/playing-cards/EvenQuads/} suggest that players begin the game by laying out 8 cards. However, our calculations show the following:

\begin{itemize}
\item An $8$-card layout will contain a quad about $80\%$ of the time.
\item A $9$-card layout will contain a quad about $97\%$ of the time.
\item A $10$-card layout will always contain at least one quad.
\end{itemize}

Thus, to guarantee a quad, you should deal out $10$ cards. For a bit more of a challenge, we suggest starting with $9$ cards. Finally, for more seasoned players, starting with $8$ cards will make it more difficult to find quads, and will require adding another card about one-fifth of the time. 

A more detailed description of card layouts in Quad-16, Quad-32, Quad-64, and Quad-128 is forthcoming.

\printbibliography

\end{document}